\documentclass[11pt]{article}

\usepackage[T1]{fontenc}
\usepackage{lmodern}
\usepackage{microtype}
\usepackage{amsmath,amssymb,mathtools,amsthm}
\usepackage{geometry}
\usepackage{enumitem}
\usepackage{graphicx}
\usepackage{tikz}
\usepackage{pgfplots}
\usepackage{aliascnt}
\usepackage[colorlinks=true,linkcolor=blue!55!black,citecolor=blue!55!black,urlcolor=blue!55!black]{hyperref}
\usepackage[nameinlink,capitalise]{cleveref}
\usepackage{verbatim}

\pgfplotsset{compat=1.18}
\allowdisplaybreaks

\newtheorem{theorem}{Theorem}
\newaliascnt{corollary}{theorem}
\newtheorem{corollary}[corollary]{Corollary}
\aliascntresetthe{corollary}
\newaliascnt{lemma}{theorem}
\newtheorem{lemma}[lemma]{Lemma}
\aliascntresetthe{lemma}
\newaliascnt{proposition}{theorem}
\newtheorem{proposition}[proposition]{Proposition}
\aliascntresetthe{proposition}
\theoremstyle{definition}
\newaliascnt{definition}{theorem}

\aliascntresetthe{definition}
\theoremstyle{remark}
\newaliascnt{remark}{theorem}

\aliascntresetthe{remark}

\crefname{theorem}{Theorem}{Theorems}
\Crefname{theorem}{Theorem}{Theorems}
\crefname{corollary}{Corollary}{Corollaries}
\Crefname{corollary}{Corollary}{Corollaries}
\crefname{lemma}{Lemma}{Lemmas}
\Crefname{lemma}{Lemma}{Lemmas}
\crefname{proposition}{Proposition}{Propositions}
\Crefname{proposition}{Proposition}{Propositions}
\crefname{definition}{Definition}{Definitions}
\Crefname{definition}{Definition}{Definitions}
\crefname{remark}{Remark}{Remarks}
\Crefname{remark}{Remark}{Remarks}

\newcommand{\PG}{\mathrm{PG}}
\newcommand{\Fq}{\mathbb F_q}
\newcommand{\cL}{\mathcal L}

\newcommand{\cG}{\mathcal G}
\newcommand{\cS}{\mathcal S}

\newcommand{\qbinom}[2]{\genfrac{[}{]}{0pt}{}{#1}{#2}_{q}}
\newcommand{\incdev}[2]{\left|I(#1,#2)-q^{-d}|#1|\,|#2|\right|}

\title{A combinatorial approach to point–flat incidences over finite fields}

\author{Ben Lund\thanks{Email: \texttt{lund.ben@gmail.com}.}\quad and Tao Zhang\thanks{Email: \texttt{zhant220@163.com}.}\\
\small Institute of Mathematics and Interdisciplinary Sciences,\\[-1mm]
\small Xidian University, Xi'an 710126, China}
\date{}

\begin{document}
\maketitle

\begin{abstract}
We establish new bounds for incidences between a point set \(P\) and a family \(L\) of \(n\)-flats in \(\operatorname{PG}(n+d,q)\). For fixed dimensions, our bound on the incidence discrepancy has an explicit piecewise-linear exponent in \(\log_q|L|\), improving the classical estimate of Haemers and the Kong–Tamo bound in specified ranges of the number of flats. Matching constructions establish sharpness up to constant factors in several parameter ranges. The proof is combinatorial and avoids spectral and Fourier analytic methods.  As applications, we obtain improved estimates for rich flats and exceptional orthogonal projections, together with stronger lower bounds for Furstenberg sets in certain ranges where the fraction of prescribed directions is small.
\end{abstract}
\section{Introduction}

In this paper, we obtain an improved upper bound on the number of incidences between points and projective flats.
Our bound is relevant when the cardinalities of the point set and the family of flats are large relative to \(q\). It improves on known bounds when the codimension exceeds one and the family of flats is sufficiently small, in the ranges specified below.

We work over a finite field $\mathbb{F}_q$, where $q$ is a power of a prime, and study incidences between points and $n$-flats in $\PG(n+d,q)$.
Throughout the paper, $q$ is the asymptotic parameter, while the
dimensions $n$ and $d$ are treated as constants. We write
\begin{alignat*}{2}
 f(q) &\precsim g(q)
 &\qquad \text{if}\qquad
 &f(q)\leq (1+o(1))g(q), \\
 f(q) &\ll g(q)
 &\qquad \text{if}\qquad
 &f(q)=O(g(q)), \text{ and} \\
 f(q) &\asymp g(q)
 &\qquad \text{if}\qquad
 &f(q)\ll g(q)\text{ and }g(q)\ll f(q).
\end{alignat*}
In all cases, the implied constants may depend on $n$ and $d$, and never on $q$. If $P$
is a set of points and $\cL$ is a set of $n$-flats in
$\PG(n+d,q)$, then $I(P,\cL)$ denotes the number of incident pairs
$(x,\ell)\in P\times\cL$.

\cref{thm:vinh} is the prototypical example of an incidence bound of the type that we study here.

\begin{theorem}\label{thm:vinh}
Let $P$ be a set of points and let $\cL$ be a set of $n$-flats, both in
$\PG(n+d,q)$. Then
\begin{equation}\label{eq:vinh}
 \incdev{P}{\cL}
 \precsim q^{dn/2}|P|^{1/2}|\cL|^{1/2}.
\end{equation}
\end{theorem}

\cref{thm:vinh} was first proved (in the more general setting of block designs) by Haemers; see \cite[Theorem 5.1]{Haemers1995} and
\cite[Theorem 3.1.1]{Haemers1979}.
The case $d=1$ was rediscovered by Vinh in a very influential paper \cite{vinh2011szemeredi}, and Vinh's proof was generalized by Lund and Saraf \cite{LundSaraf2016}.

The bound in \eqref{eq:vinh} is best possible as a function of
$|P||\cL|$; see, for example, the constructions in
\cite{fraser2025exceptional}. Kong and Tamo obtained a stronger estimate
for affine flats when $|\cL|$ is not too large
\cite[Theorem~1.8]{kong2025point}. The following is a projective
version of their bound; we give a proof in \cref{sec:classical-proofs}.

\begin{theorem}[Projective version of the Kong--Tamo bound]\label{thm:kong-tamo}
Let $P$ be a set of points and let $\cL$ be a set of $n$-flats, both in
$\PG(n+d,q)$. Then
\begin{equation*}
 \incdev{P}{\cL}
 \ll q^{n/2}|P|^{1/2}|\cL|^{1/2}
 \bigl(1+q^{-1}|\cL|\bigr)^{1/2}.
\end{equation*}
\end{theorem}

Comparing the two exponents shows that \cref{thm:kong-tamo} improves
\cref{thm:vinh} when $|\cL|\ll q^{dn-n+1}$. The total number of
$n$-flats in $\PG(n+d,q)$ is
$ \qbinom{n+d+1}{n+1}=(1+o(1))q^{(n+1)d}$.

Before stating our general result, we record two simpler consequences.
Matching constructions, in the precise ranges described below, show that these bounds are sharp up to constant factors.

\begin{corollary}[Standard-range bound]\label{cor:standard}
Suppose that $d\geq n$. Let $P$ be a set of points and let $\cL$ be a
set of $n$-flats in $\PG(n+d,q)$. If
\[
 q^{id}\leq|\cL|\leq q^{(i+1)d},
 \qquad
 0\leq i\leq\min\{n,d-n\},
\]
then
\begin{equation*}
 \incdev{P}{\cL}
 \ll
 q^{(id+n-i)/2}|P|^{1/2}|\cL|^{1/2}
 +q^{(n-i-1)/2}|P|^{1/2}|\cL|.
\end{equation*}
\end{corollary}

\cref{cor:dual-standard} is not the projective dual of \cref{cor:standard}; rather, its proof uses projective duality.

\begin{corollary}[Dual standard-range bound]\label{cor:dual-standard}
Suppose that $n\geq d-2$. Let $P$ be a set of points and let $\cL$ be a
set of $n$-flats in $\PG(n+d,q)$. If
\[
 q^{i(n+1)}\leq|\cL|\leq q^{(i+1)(n+1)},
 \qquad
 0\leq i\leq\min\{d-1,n+2-d\},
\]
then
\begin{equation*}
 \incdev{P}{\cL}
 \ll
 q^{(i+1)n/2}|P|^{1/2}|\cL|^{1/2}
 +q^{(n-i-1)/2}|P|^{1/2}|\cL|.
\end{equation*}
\end{corollary}

For $i=0$, the two terms in either corollary recover the two regimes of
the Kong--Tamo estimate. If $d\geq2n$, then \cref{cor:standard}
improves \cref{thm:vinh} whenever $|\cL|\ll q^{dn}$. If
$n\geq2d-3$, then \cref{cor:dual-standard} improves
\cref{thm:vinh} whenever
$|\cL|\ll q^{(n+1)(d-1)}$.

These bounds are sharp in the following sense: for every prescribed integer
$M$ in the stated range, there exist a family $\cL$ of exactly $M$
$n$-flats and a nonempty point set $P$, whose cardinality may depend
on $M$, such that $\incdev{P}{\cL}$ has the order of the asserted
upper bound. The comparison constants depend only on $n$ and $d$.
For $d\geq2$, the constructions in
\cref{prop:lower-bound-constructions} establish sharpness throughout
\[
 \begin{aligned}
  1\leq M&\leq q^{d\min\{n,d-n+1\}}
    &&\text{for \cref{cor:standard}, when $d\geq n$},\\
  1\leq M&\leq q^{(n+1)\min\{d-1,n+3-d\}}
    &&\text{for \cref{cor:dual-standard}, when $n\geq d-2$}.
 \end{aligned}
\]
In particular, when $n\geq2d-3$ and $d\geq2$, the dual-standard
bound is sharp throughout $1\leq M\leq q^{(n+1)(d-1)}$.

There is a further extension in the standard case. Let $\cG$ be the
set of all $n$-flats in $\PG(n+d,q)$. If $d\geq2n$, then the
minimum of the bounds in \cref{cor:standard} and \cref{thm:vinh} is
sharp for every integer $1\leq M<|\cG|/2$, for all sufficiently large
$q$. The first part of \cref{prop:lower-bound-constructions} covers
$M\leq q^{nd}$, and its additional assertion covers
$q^{nd}\leq M<|\cG|/2$.

In both \cref{cor:standard} and \cref{cor:dual-standard}, the exponent in the error term is a piecewise linear function of $\log_q |\cL|$ that alternates sections with slope $1/2$ and $1$.
The general result has the same basic shape, but there may be additional short pieces, some of which have slope $3/2$.

\cref{fig:standard-comparison,fig:exceptional-comparison} illustrate these bounds in two special cases. \cref{fig:standard-comparison} shows a case in which $d>2n$, where our bounds are relatively simple and always sharp (in the sense discussed above). \cref{fig:exceptional-comparison} shows a more complicated situation, including segments with slope $3/2$ and shifted short segments.

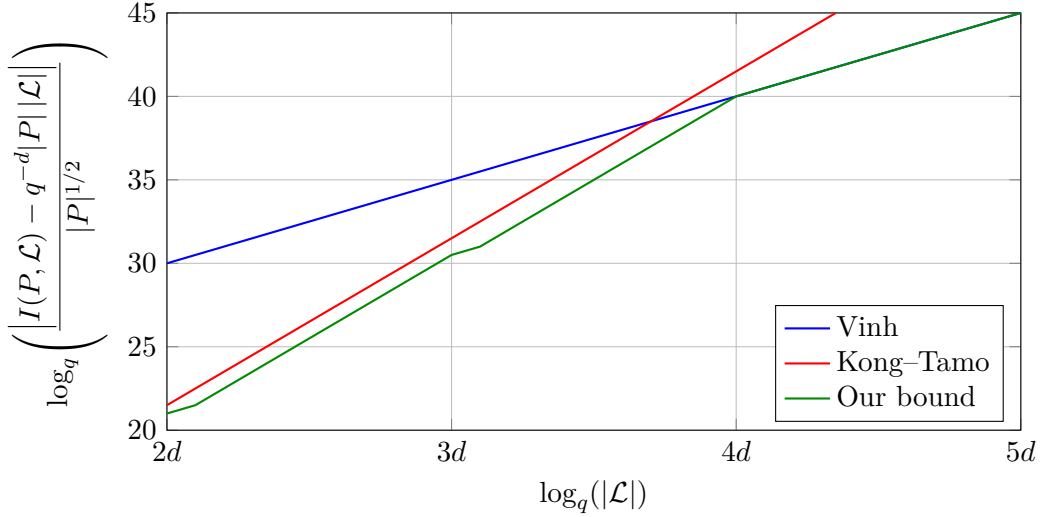
\begin{figure}[h]
\centering
\begin{tikzpicture}
\begin{axis}[
 width=.78\textwidth,
 height=.43\textwidth,
 xmin=20,xmax=50,
 ymin=20,ymax=45,
 xtick={20,30,40,50},
 xticklabels={$2d$,$3d$,$4d$,$5d$},
 ytick={20,25,30,35,40,45},
 xlabel={$\log_q(|\cL|)$},
 ylabel={$\log_q\!\left(\dfrac{\incdev{P}{\cL}}{|P|^{1/2}}\right)$},
 grid=major,
 legend style={at={(0.98,0.02)},anchor=south east,draw=black,fill=white},
 legend cell align=left,
 clip=true
]
\addplot[blue,thick,mark=none] coordinates {(20,30) (50,45)};
\addlegendentry{Vinh}
\addplot[red,thick,mark=none] coordinates {(20,21.5) (43.5,45)};
\addlegendentry{Kong--Tamo}
\addplot[green!55!black,thick,mark=none] coordinates
 {(20,21) (21,21.5) (30,30.5) (31,31) (40,40) (50,45)};
\addlegendentry{Our bound}
\end{axis}
\end{tikzpicture}
\caption{Comparison of the incidence-error exponents for $d=10$ and
$n=4$.}
\label{fig:standard-comparison}
\end{figure}

\begin{figure}[h]
\centering
\begin{tikzpicture}
\begin{axis}[
 width=15.5cm,
 height=9.5cm,
 xmin=6,xmax=32,
 ymin=7.5,ymax=31,
 xtick={6,12,18,24,30},
 xticklabels={$d$,$2d$,$3d$,$4d$,$5d$},
 ytick={8,12,16,20,24,28,31},
 xlabel={$m=\log_q|\mathcal L|$},
 ylabel={$\log_q\!\left(
   \dfrac{\left|I(P,\mathcal L)-q^{-d}|P|\,|\mathcal L|\right|}
         {|P|^{1/2}}
   \right)$},
 grid=major,
 legend style={
   at={(0.98,0.02)},
   anchor=south east,
   draw=black,
   fill=white
 },
 legend cell align=left,
 clip=true
]
\addplot[blue,thick,mark=none]
 coordinates {(6,18) (32,31)};
\addlegendentry{Vinh}

\addplot[red,thick,mark=none]
 coordinates {(6,8) (29,31)};
\addlegendentry{Kong--Tamo}

\addplot[green!55!black,thick,mark=none]
 coordinates {
   (6,8) (7,8.5) (13,14.5) (14,15)
   (20,21) (21,21.5) (28,28.5) (30,30) (32,31)
 };
\addlegendentry{Our bound}

\addplot[orange!85!black,thick,mark=none]
 coordinates {
   (6,8) (7,8.5) (12,13.5) (13,14)
   (18,19) (19,19.5) (24,24.5) (25,25) (30,30) (32,31)
 };
\addlegendentry{Lower bound (Prop.~6)}
\end{axis}
\end{tikzpicture}
\caption{Comparison of the incidence-error exponents for $d=6$ and $n=5$.}
\label{fig:exceptional-comparison}
\end{figure}
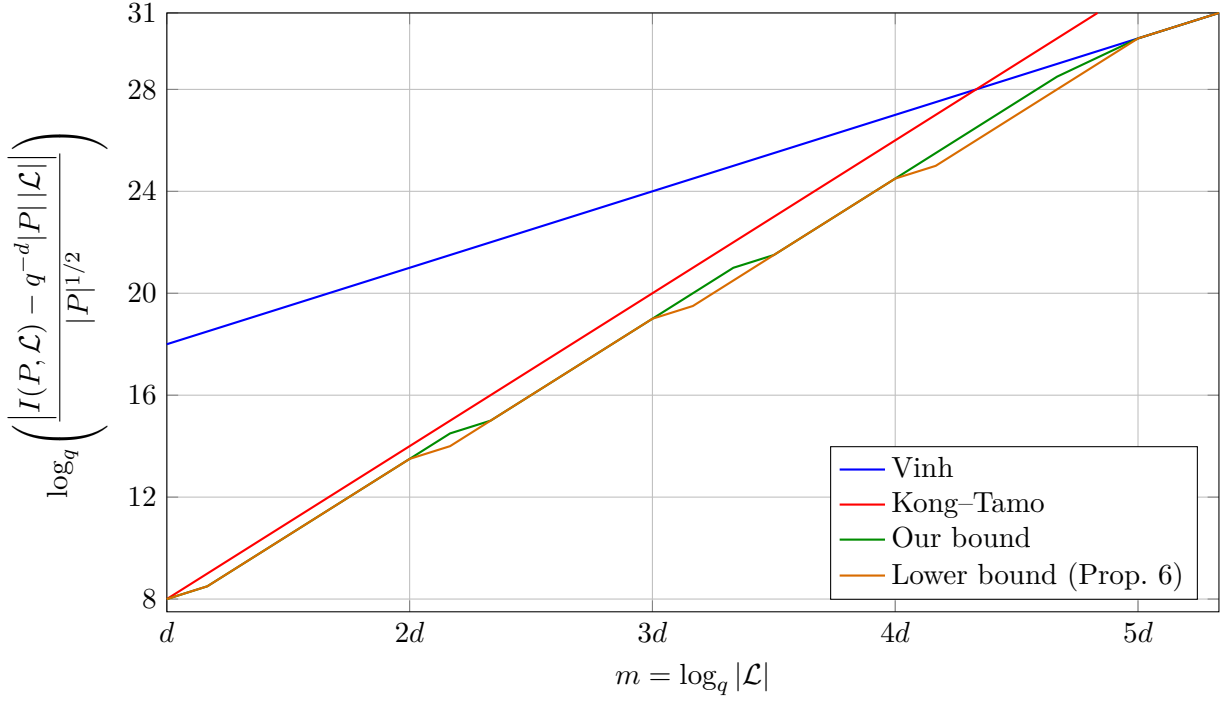

To describe the general case precisely, for a nonempty family $\cL$ write
\[
 m=\log_q|\cL|,
 \qquad m\geq0,
\]
and introduce the normalized parameters
\begin{equation}\label{eq:normalization-intro}
 (\nu,\delta,\rho)=
 \begin{cases}
  (n,d,0),& d>n,\\
  (d-1,n+1,n-d+1),& d\leq n.
 \end{cases}
\end{equation}
Thus
\begin{equation*}
 (\nu+1)\delta=(n+1)d,
 \qquad
 \nu+\rho=n,
 \qquad
 \delta>\nu.
\end{equation*}
The second line of \eqref{eq:normalization-intro} comes from projective
duality. 
If \(d\le n\) and two \(n\)-flats intersect in a \(j\)-flat, then their dual \((d-1)\)-flats intersect in a \((j-\rho)\)-flat.

The upper bound in the general case takes the form
\begin{equation*}
 \incdev{P}{\cL}
 \ll
 |P|^{1/2}
 q^{\frac12\min\{m+nd,\,F_{\nu,\delta}(m)+\rho\}},
\end{equation*}
where $F_{\nu,\delta}$ is a continuous piecewise-linear function. Its rough
qualitative shape is summarized in \cref{tab:F-shape}. Apart from a
final slope-$3/2$ piece which never affects the final incidence bound since we take a minimum with the bound from \cref{thm:vinh}, the graph consists mainly of pieces of slope $2$, separated by short pieces of slope $1$ or $3/2$.

\begin{table}[ht]
\centering
\renewcommand{\arraystretch}{1.25}
\begin{tabular}{c|c|c}
 & breakpoints & short pieces \\ \hline
$\delta\geq2\nu$
 & regularly spaced
 & slope $1$ \\
$\nu<\delta<2\nu$
 & some shifted
 & slope $1$ or $3/2$
\end{tabular}
\caption{Qualitative shape of the exponent $F_{\nu,\delta}$.}
\label{tab:F-shape}
\end{table}

We next define $F_{\nu,\delta}$ on $[0,\infty)$, describing its
breakpoints first on $[0,(\nu+1)\delta]$. If $\nu=0$, set $F_{0,\delta}(m)=m.$
Suppose henceforth that $\nu\geq1$, and put $r=2\nu-\delta.$
For $0\leq j\leq\nu-1$, define the correction
\begin{equation*}
 c_j=
 \max_{\substack{0\leq u\leq j+1\\u\in\mathbb Z}}
 u(r-j+1-u).
\end{equation*}
Set
\begin{equation*}
 b_\nu=0,
 \qquad
 b_j=(\nu-j)\delta+c_j
 \quad(0\leq j\leq\nu-1),
\end{equation*}
and, for $1\leq j\leq\nu$, define the width
\begin{equation*}
 w_j=
 \begin{cases}
  2,& r\geq3j,\\
  1,& r<3j.
 \end{cases}
\end{equation*}
Since $r<\nu$, one always has $w_\nu=1$.

The numbers $b_j$ are the principal breakpoints of
$F_{\nu,\delta}$. For $1\leq j\leq\nu$, define
\begin{equation}\label{eq:AB-def}
 A_j=(b_j,2b_j+j),
 \qquad
 B_j=\bigl(b_j+w_j,\,2(b_j+w_j)+j-1\bigr),
\end{equation}
and set $ A_0=(b_0,2b_0).$
On $0\leq m\leq b_0$, the graph of $F_{\nu,\delta}$ is the polygonal
path joining
\begin{equation*}
 A_\nu,\ B_\nu,\ A_{\nu-1},\ B_{\nu-1},\ldots,
 A_1,\ B_1,\ A_0
\end{equation*}
in this order. The segment from $A_j$ to $B_j$ has slope
\[
 2-\frac1{w_j}
 =
 \begin{cases}
  1,& w_j=1,\\[1mm]
  \frac32,& w_j=2,
 \end{cases}
\]
while every segment from $B_j$ to $A_{j-1}$ has slope $2$.

It remains to describe the terminal part of the graph. Put
\begin{equation*}
 b_\ast=b_0+\max\{0,\delta-2\nu\}.
\end{equation*}
Then
\begin{align*}
 F_{\nu,\delta}(m)
 &=m+b_0,
 &&b_0\leq m\leq b_\ast,\\
 F_{\nu,\delta}(m)
 &=b_\ast+b_0+\frac32(m-b_\ast),
 &&b_\ast\leq m\leq(\nu+1)\delta.
\end{align*}
When $\delta\leq2\nu$, the first interval degenerates to the single point $b_0$.

We can now state the general result.

\begin{theorem}[Main theorem]
\label{thm:main}
Let $n,d\geq1$, let $P$ be a set of points, and let $\cL$ be a nonempty set of
$n$-flats in $\PG(n+d,q)$. Write $|\cL|=q^m$, define
$(\nu,\delta,\rho)$ by \eqref{eq:normalization-intro}, and define
$F_{\nu,\delta}$ as above. Then
\begin{equation}\label{eq:main-incidence}
 \incdev{P}{\cL}
 \ll
 |P|^{1/2}q^{\Omega_{n,d}(m)/2},
 \qquad
 \Omega_{n,d}(m)
 =\min\bigl\{m+nd,\,\rho+F_{\nu,\delta}(m)\bigr\}.
\end{equation}
\end{theorem}

The simple bounds in \cref{cor:standard,cor:dual-standard} correspond
to the regularly spaced portions of this polygonal curve. In
particular, when $\delta\geq2\nu$, all of the corrections $c_j$ vanish
and all of the widths $w_j$ equal one. Thus the successive short pieces
have slope $1$, the intervening pieces have slope $2$, and the
breakpoints are
\[
 b_j=(\nu-j)\delta.
\]
When $\delta>2\nu$, these pieces are followed by an additional
slope-$1$ segment of length $\delta-2\nu$, and then by the final
slope-$3/2$ segment.

The range $\nu<\delta<2\nu$ is the exceptional regime. Here
$r=2\nu-\delta>0$. The corrections $c_j$ shift some of the breakpoints
away from $(\nu-j)\delta$, while the condition $3j\leq r$ gives
$w_j=2$ and hence replaces the corresponding short slope-$1$ piece by
a slope-$3/2$ piece. Thus the extra complexity in the general formula
is governed by two effects: a displacement of certain breakpoints and
a change in slope on certain short pieces.

The full formula may be sharper than either two-term consequence even
in relatively small dimensions. For example, when $(n,d)=(3,4)$ and
$m=10$, the general bound gives
$F_{3,4}(10)=20$,
whereas the corresponding two-term estimate has exponent $21$.

\cref{thm:main} improves the \cref{thm:vinh} bound whenever $m \leq \nu \delta - j_{\mathrm V}$, where
\[j_{\mathrm V} = \max\{0 \leq j \leq \nu-1: c_j \geq j(\delta - 1) \}.\]
In the case $\delta \geq 2 \nu$, we have $j_{\mathrm V} = 0$ and hence obtain an improvement whenever $m \leq \nu \delta$. For $\nu\geq1$, the first part of \cref{prop:lower-bound-constructions} gives matching lower bounds throughout $0\leq m\leq\nu\delta$; the additional standard-case range is described above.

\subsection{Proof ideas and outline}

All of the theorems discussed above can be interpreted as statements about the pseudorandomness of the  incidence graph between points and $n$-flats in $\PG(n+d,q)$.
\cref{thm:vinh} is essentially (up to lower order terms) the claim that this graph is {\em $(q^{-d}, q^{dn/2})$-bijumbled}, while the remaining theorems give finer characterizations that are more precisely tailored to this specific graph.
Earlier proofs of \cref{thm:vinh} and \cref{thm:kong-tamo} rely on spectral graph theory \cite{Haemers1995, vinh2011szemeredi, LundSaraf2016}, or (similarly) on finite Fourier analysis \cite{fraser2025exceptional}.
By contrast, our proofs work directly with incidence and intersection counts, without using eigenvalue estimates or Fourier analysis.
Instead of computing eigenvalues or Fourier coefficients, we bound the codegrees of pairs of vertices in the incidence graph.
This is a standard approach to pseudorandom graphs, going back to the foundational work of Thomason \cite{thomason1989dense}.

The proof is based on two applications of Cauchy–Schwarz. The first bounds the squared incidence discrepancy in terms of weighted counts of pairs of flats with prescribed intersection dimensions. The second gives a finite system of quadratic inequalities for these counts, with coefficients given by common-neighbor kernels. We calculate the leading powers of \(q\) in these kernels. Projective duality reduces the analysis to the normalized range \(\delta>\nu\), where passing to logarithmic coordinates produces a monotone, contractive map. Its unique fixed point gives upper bounds for the exponents of the weighted pair counts. We then construct explicit supersolutions and prove matching lower bounds within the fixed-point system to evaluate the maximum of these coordinates and the diagonal contribution. The function \(F_{\nu,\delta}\) records this scalar maximum.

\subsection{Organization}

\Cref{sec:lowerBounds} gives the star and contained-flat constructions
and establishes the sharpness claims above.
\Cref{sec:applications} applies the incidence bound to rich flats,
exceptional orthogonal projections, and Furstenberg sets.
\Cref{sec:graph-lemma} proves the graph-theoretic variance estimate
used in both Cauchy--Schwarz steps, and
\cref{sec:classical-proofs} uses it to give new proofs of
\cref{thm:vinh,thm:kong-tamo}.
\Cref{sec:initial-bounds} derives the pair-counting inequalities and
the common-neighbor kernel asymptotics.
Finally, \cref{sec:main-proof} bounds the pair counts through a
fixed-point system, evaluates the resulting scalar exponent, and
deduces \cref{thm:main} and
\cref{cor:standard,cor:dual-standard}.

\section{Lower bounds}\label{sec:lowerBounds}
In this section, we present two simple constructions that attain the upper bounds given in \cref{cor:standard} and \cref{cor:dual-standard}.
When $d>n$, we use stars.
The star of $n$-flats that contain a fixed $(n-i)$-flat has $q^{id}$ members.
Taking several parallel copies of this star allows us to vary the size
of $\cL$.
When only a few copies are needed, the affine parts of the
cores give the high-multiplicity points; for larger collections, the
common $(n-i-1)$-flat at infinity gives the larger contribution.

When $d\leq n$, we use a dual contained-flat construction.  
The family of $n$-flats contained in a  fixed
$(n+i)$-flat has $q^{i(n+1)}$ members.  
Again, several parallel copies give the first term in the bound, while the second term is
obtained by taking an arbitrary collection of $n$-flats inside one
$(n+i+1)$-flat.  

For an integer $M$, define 
\begin{equation*}
 \mathcal E_{n,d}(M)
 =
 \max_{\substack{|\cL|=M\\
                  \varnothing\neq P\subseteq\PG(n+d,q)}}
 \frac{I(P,\cL)-q^{-d}|P|M}{|P|^{1/2}},
\end{equation*}
where the maximum is over sets $\cL$ of $n$-flats and sets $P$ of
points.  Recall that $(\nu,\delta,\rho)$ are the normalized parameters
from \eqref{eq:normalization-intro}; thus
\[
 (\nu,\delta)=
 \begin{cases}
  (n,d),& d>n,\\
  (d-1,n+1),& d\leq n.
 \end{cases}
\]

\begin{proposition}[Star and contained-flat constructions]
\label{prop:lower-bound-constructions}
Suppose that $\nu\geq1$.  Let $0\leq i\leq\nu-1$ and
\[
 q^{i\delta}\leq M\leq q^{(i+1)\delta}.
\]
Then
\begin{equation}\label{eq:construction-lower-bound}
 \mathcal E_{n,d}(M)
 \gg
 q^{(i\delta+n-i)/2}M^{1/2}
 +q^{(n-i-1)/2}M.
\end{equation}
Equivalently, if $M=q^m$, then
\begin{equation}\label{eq:construction-exponent}
 \mathcal E_{n,d}(q^m)
 \gg
 q^{\frac12\max\{m+i\delta+n-i,\,2m+n-i-1\}}.
\end{equation}
When $d>n$, the two terms are realized by star constructions.  When
$d\leq n$, they are realized by the dual contained-flat constructions.

In addition, suppose that $d>n$, and let $\cG$ be the set of all
$n$-flats in $\PG(n+d,q)$. For every integer
\[
 q^{nd}\leq M<\tfrac12|\cG|,
\]
we have, for all sufficiently large $q$,
\begin{equation}\label{eq:construction-terminal-bound}
 \mathcal E_{n,d}(M)\gg q^{nd/2}M^{1/2}.
\end{equation}
\end{proposition}

\begin{proof}
Set
\[
 A_i(M)=q^{(i\delta+n-i)/2}M^{1/2},
 \qquad
 B_i(M)=q^{(n-i-1)/2}M.
\]
Since
\begin{equation}\label{eq:A-over-B}
 \frac{A_i(M)}{B_i(M)}
 =\left(\frac{q^{i\delta+1}}{M}\right)^{1/2},
\end{equation}
$A_i(M)$ is the larger term for $M\leq q^{i\delta+1}$, whereas
$B_i(M)$ is the larger term for $M\geq q^{i\delta+1}$.  It therefore suffices to construct examples attaining the larger term in each of the two ranges.

\medskip
\noindent
\emph{The case $d>n$.}
Here $\nu=n$ and $\delta=d$.  Work in an affine chart of
$\PG(n+d,q)$ with coordinates
\[
 (z,x,y)\in
 \mathbb F_q^{\,n-i}\times
 \mathbb F_q^{\,i}\times
 \mathbb F_q^{\,d}.
\]
For $A\in\operatorname{Mat}_{d\times i}(\mathbb F_q)$ and
$b\in\mathbb F_q^d$, let $\ell_{A,b}$ be the projective closure of
\begin{equation*}
 \{(z,x,Ax+b):
   z\in\mathbb F_q^{\,n-i},\ x\in\mathbb F_q^i\}.
\end{equation*}
For fixed $b$, the $q^{id}$ flats $\ell_{A,b}$ form an affine cell in
the star with core equal to the projective closure of
\[
 C_b^{\circ}
 =\{(z,0,b):z\in\mathbb F_q^{\,n-i}\}.
\]

First suppose that $q^{id}\leq M\leq q^{id+1}$, and put
\[
 R=\left\lceil\frac{M}{q^{id}}\right\rceil.
\]
Thus $1\leq R\leq q$.  Choose distinct
$b_1,\ldots,b_R\in\mathbb F_q^d$, choose exactly $M$ flats from the
corresponding $R$ cells,
and set
\[
 P=\bigcup_{r=1}^R C_{b_r}^{\circ}.
\]
The sets $C_{b_r}^{\circ}$ are pairwise disjoint, so
\[
 |P|=Rq^{n-i}.
\]
Moreover, each selected flat from the $r$th cell contains all
$q^{n-i}$ points of $C_{b_r}^{\circ}$ and no point of
$C_{b_s}^{\circ}$ for $s\neq r$.  Hence
\[
 I(P,\cL)=Mq^{n-i}.
\]
Consequently,
\begin{align}
 \frac{I(P,\cL)-q^{-d}|P|M}{|P|^{1/2}}
 &=
 (1-Rq^{-d})\frac{Mq^{(n-i)/2}}{R^{1/2}}.
 \label{eq:star-core-discrepancy}
\end{align}
Since $d>n\geq1$ and $R\leq q$, the factor $1-Rq^{-d}$ is bounded
below by a positive absolute constant.  Also
\[
 R\leq 2\frac{M}{q^{id}},
\]
so \eqref{eq:star-core-discrepancy} gives
\begin{equation*}
 \frac{I(P,\cL)-q^{-d}|P|M}{|P|^{1/2}}
 \gg q^{(id+n-i)/2}M^{1/2}=A_i(M).
\end{equation*}

For the second term, observe that all flats of the form $\ell_{A,b}$
contain the same $(n-i-1)$-flat $K$ at infinity, namely the flat of
directions of the $z$-coordinate subspace.  Thus
\[
 |K|=\frac{q^{n-i}-1}{q-1}\asymp q^{n-i-1}.
\]
There are $q^{(i+1)d}$ flats of the form $\ell_{A,b}$.  Hence, for any
$M\leq q^{(i+1)d}$, choose any $M$ of them and take $P=K$.  Then
$I(P,\cL)=|P|M$, and therefore
\begin{align*}
 \frac{I(P,\cL)-q^{-d}|P|M}{|P|^{1/2}}
 &=(1-q^{-d})M|P|^{1/2}\notag\\
 &\asymp q^{(n-i-1)/2}M=B_i(M).
\end{align*}
Together with \eqref{eq:A-over-B}, this proves
\eqref{eq:construction-lower-bound} when $d>n$.

To prove \eqref{eq:construction-terminal-bound}, use the same
construction with $i=n$, so that the affine cores are the single
points $(0,b)$. Put
\[
 R=\left\lceil\frac{M}{q^{nd}}\right\rceil.
\]
Since
\[
 |\cG|=\qbinom{n+d+1}{n+1}
       =(1+O(q^{-1}))q^{(n+1)d},
\]
the assumption on $M$ gives
\[
 \frac{R}{q^d}\leq\frac12+O(q^{-1}).
\]
Thus, for all sufficiently large $q$, we can choose distinct
$b_1,\ldots,b_R\in\mathbb F_q^d$. Select exactly $M$ flats from the
corresponding cells and put $P=\{(0,b_1),\ldots,(0,b_R)\}$.
Every selected flat contains exactly one point of $P$, so
$|P|=R$ and $I(P,\cL)=M$. Using also $R\leq2M/q^{nd}$, we obtain
\[
 \frac{I(P,\cL)-q^{-d}|P|M}{|P|^{1/2}}
 =\left(1-\frac{R}{q^d}\right)\frac{M}{\sqrt R}
 \gg q^{nd/2}M^{1/2},
\]
as required.

\medskip
\noindent
\emph{The case $d\leq n$.}
Here $\nu=d-1$ and $\delta=n+1$, so $i\leq d-2$.  Work in an affine
chart with coordinates
\[
 (x,y_1,\ldots,y_d)\in\mathbb F_q^n\times\mathbb F_q^d.
\]
For $b\in\mathbb F_q$, let
\[
 H_b=\{(x,y):y_{i+1}=b,\ y_{i+2}=\cdots=y_d=0\}.
\]
This is an affine $(n+i)$-flat.  If
$f_1,\ldots,f_i:\mathbb F_q^n\to\mathbb F_q$ are affine functions,
let $\ell_{f,b}$ be the projective closure of
\[
 \{(x,f_1(x),\ldots,f_i(x),b,0,\ldots,0):x\in\mathbb F_q^n\}.
\]
For each fixed $b$, there are exactly $q^{i(n+1)}$ such projective $n$-flats,
all contained in the projective closure of $H_b$.

First suppose that
$q^{i(n+1)}\leq M\leq q^{i(n+1)+1}$, and put
\[
 R=\left\lceil\frac{M}{q^{i(n+1)}}\right\rceil.
\]
Again $1\leq R\leq q$.  Choose distinct
$b_1,\ldots,b_R\in\mathbb F_q$, choose exactly $M$ flats from the
corresponding $R$ collections, with at least one from each, and let
$P$ be the union of the affine point sets of
$H_{b_1},\ldots,H_{b_R}$.  These affine flats are pairwise disjoint,
so
\[
 |P|=Rq^{n+i},
 \qquad
 I(P,\cL)=Mq^n.
\]
It follows that
\begin{equation}\label{eq:dual-union-discrepancy}
 \frac{I(P,\cL)-q^{-d}|P|M}{|P|^{1/2}}
 =
 (1-Rq^{i-d})\frac{Mq^{(n-i)/2}}{R^{1/2}}.
\end{equation}
Since $i\leq d-2$ and $R\leq q$,
\[
 1-Rq^{i-d}\geq1-q^{-1}\gg1.
\]
Also
\[
 R\leq 2\frac{M}{q^{i(n+1)}}.
\]
Therefore \eqref{eq:dual-union-discrepancy} gives
\begin{align*}
 \frac{I(P,\cL)-q^{-d}|P|M}{|P|^{1/2}}
 &\gg q^{\{i(n+1)+n-i\}/2}M^{1/2}\notag\\
 &=q^{(i+1)n/2}M^{1/2}\notag\\
 &=q^{(i\delta+n-i)/2}M^{1/2}
 =A_i(M).
\end{align*}

For the second term, let
\[
 W=\{(x,y):y_{i+2}=\cdots=y_d=0\}.
\]
This is an affine $(n+i+1)$-flat.  The graphs of arbitrary
$(i+1)$-tuples of affine functions $\mathbb F_q^n\to\mathbb F_q$
give exactly $q^{(i+1)(n+1)}$ distinct $n$-flats contained in $W$.
Thus, for any $M\leq q^{(i+1)(n+1)}$, choose any $M$ of these flats
and take $P$ to be the affine point set of $W$.  Then
\[
 |P|=q^{n+i+1},
 \qquad
 I(P,\cL)=Mq^n,
\]
and consequently
\begin{align*}
 \frac{I(P,\cL)-q^{-d}|P|M}{|P|^{1/2}}
 &=(1-q^{i+1-d})q^{(n-i-1)/2}M\notag\\
 &\gg q^{(n-i-1)/2}M=B_i(M).
 \label{eq:dual-second-term}
\end{align*}
This proves \eqref{eq:construction-lower-bound} when $d\leq n$.
Finally, \eqref{eq:A-over-B} shows that in either case one of the two
constructions attains the larger of $A_i(M)$ and $B_i(M)$; this is at
least half their sum.  Substituting $M=q^m$ gives
\eqref{eq:construction-exponent}.
\end{proof}

When $\delta\geq2\nu\geq2$, the lower bound in
\cref{prop:lower-bound-constructions} agrees, up to the implied
constant, with the upper bound in
\cref{thm:main} throughout
\[
 1\leq M\leq q^{\nu\delta}.
\]
In the original parameters, the condition $\delta \geq 2 \nu \geq 2$ is $d\geq2n$ for the star
construction and $n\geq2d-3$ for the contained-flat construction.
Thus the standard bound is sharp throughout $1\leq M\leq q^{nd}$
when $d\geq2n$, and the dual-standard bound is sharp throughout
$1\leq M\leq q^{(n+1)(d-1)}$ when $n\geq2d-3$ and $d\geq2$.
For $d\geq2n$, \eqref{eq:construction-terminal-bound} also matches
\cref{thm:vinh} whenever $q^{nd}\leq M<|\cG|/2$, where $\cG$
is the set of all $n$-flats. Consequently, the minimum of
\cref{cor:standard} and \cref{thm:vinh} is sharp for every integer
$1\leq M<|\cG|/2$, for all sufficiently large $q$.

\section{Applications}\label{sec:applications}

\subsection{Rich flats}

First, we prove a bound on the number of rich flats.
For a nonempty set of points $P\subseteq\PG(n+d,q)$, put
\[
 \mathcal R_t(P)
 =\{\ell:\ell\text{ is an $n$-flat and }|\ell\cap P|\geq t\},
 \qquad R_t(P)=|\mathcal R_t(P)|,
\]
and write
\[
 \Delta=\Delta_t^P=t-q^{-d}|P|.
\]
We assume throughout this subsection that $\Delta>0$.
For comparison, we first record the standard consequence of
\cref{thm:vinh}.  
The corresponding estimate for rich lines in the
affine plane was explicitly proved by Iosevich, Rudnev, and Zhai
\cite{IosevichRudnevZhai2015}.

\begin{corollary}[Vinh-type bound for rich flats]\label{cor:rich-vinh}
With the notation above,
\begin{equation}\label{eq:rich-vinh}
 R_t(P)\precsim\frac{q^{nd}|P|}{\Delta^2}.
\end{equation}
\end{corollary}

\begin{proof}
Write $R=R_t(P)$.  Every member of $\mathcal R_t(P)$ is $t$-rich, so
\begin{equation}\label{eq:rich-basic}
 \Delta R
 \leq I(P,\mathcal R_t(P))-q^{-d}|P|R.
\end{equation}
Applying \cref{thm:vinh} gives
$\Delta R\precsim q^{nd/2}|P|^{1/2}R^{1/2}$.
If $R>0$, division by $\Delta R^{1/2}$ and squaring prove the claim;
the case $R=0$ is immediate.
\end{proof}

The next corollary of \cref{thm:main} bounds the number of rich flats
by inverting the short pieces of $F_{\nu,\delta}$.  A piece of width one
gives a bound with $\Delta^2$ in the denominator, while a piece of width
two gives a bound with $\Delta^4$ in the denominator.

\begin{samepage}
\begin{corollary}[Rich flats]\label{cor:rich-main}
There is a constant $C=C(n,d)>0$ such that, for every $1\leq j\leq\nu$,
\begin{equation}\label{eq:rich-threshold}
 \Delta\geq Cq^{(\rho+j-1)/2}|P|^{1/2}
\end{equation}
implies
\begin{equation}\label{eq:rich-main}
 R_t(P)\ll
 q^{b_j}\left(\frac{q^{\rho+j}|P|}{\Delta^2}\right)^{w_j}.
\end{equation}
In particular, the latter bound is
\begin{equation*}
 R_t(P)\ll
 \begin{cases}
  \displaystyle\frac{q^{b_j+\rho+j}|P|}{\Delta^2},&w_j=1,\\[3mm]
  \displaystyle\frac{q^{b_j+2\rho+2j}|P|^2}{\Delta^4},&w_j=2.
 \end{cases}
\end{equation*}
\end{corollary}
\end{samepage}

Subject to \eqref{eq:rich-threshold}, the ratio of the displayed bound
in \eqref{eq:rich-main} to the Vinh-type bound in
\eqref{eq:rich-vinh}, with implied constants suppressed, is
\begin{equation*}
 \frac{q^{b_j}(q^{\rho+j}|P|/\Delta^2)^{w_j}}
      {q^{nd}|P|/\Delta^2}
 =
 \begin{cases}
  q^{b_j+\rho+j-nd},&w_j=1,\\[1mm]
  \displaystyle q^{b_j+2\rho+2j-nd}\frac{|P|}{\Delta^2},&w_j=2.
 \end{cases}
\end{equation*}
Thus the width-one estimate is asymptotically smaller exactly when
$b_j+\rho+j<nd$.  For width two, the displayed expression is smaller
exactly when $\Delta>q^{(b_j+2\rho+2j-nd)/2}|P|^{1/2}$;
the improvement is asymptotic when the ratio of $\Delta$ to this
threshold tends to infinity.

\begin{proof}[Proof of \cref{cor:rich-main}]
Write $R=R_t(P)$ and assume $R>0$.  Fix $1\leq j\leq\nu$, and
choose $C$ sufficiently large.
If $R\geq q^{b_j+w_j}$, take a subfamily
$\cL_0\subseteq\mathcal R_t(P)$ with $|\cL_0|=q^{b_j+w_j}$.
The identity
\[
 F_{\nu,\delta}(b_j+w_j)=2(b_j+w_j)+j-1
\]
and \cref{thm:main} give
\[
 \Delta|\cL_0|
 \ll |P|^{1/2}q^{(\rho+j-1)/2}|\cL_0|,
\]
contrary to \eqref{eq:rich-threshold}.  Thus, writing $m=\log_qR$,
we have $0\leq m<b_j+w_j$.

We use the following affine upper bound:
\begin{equation}\label{eq:rich-affine-extension}
 F_{\nu,\delta}(m)
 \leq\left(2-\frac1{w_j}\right)m+\frac{b_j}{w_j}+j,
 \qquad 0\leq m\leq b_j+w_j.
\end{equation}
To verify it, first note that the breakpoint formulas give
\[
 b_{k-1}-b_k\geq2\qquad(1\leq k\leq\nu).
\]
Indeed, the gap is $\delta$ for $k=\nu$; for $k<\nu$ it is at
least $\delta$ when $r<3k$, and equals
$2(\delta-\nu)+4k$ otherwise.  Since $w_j,w_k\leq2$, for $k>j$
we obtain
\[
 b_j-b_k\geq w_j(k-j),
 \qquad
 b_j-b_k-w_k\geq w_j(k-j-1).
\]
These inequalities say precisely that the affine function in
\eqref{eq:rich-affine-extension} lies above the vertices $A_k$ and
$B_k$.  It passes through $A_j$ and $B_j$, so linear interpolation
along the polygonal graph proves \eqref{eq:rich-affine-extension}.

Applying \cref{thm:main} to \eqref{eq:rich-basic}, and then using
\eqref{eq:rich-affine-extension}, gives
\[
 \Delta R
 \ll |P|^{1/2}q^{(\rho+j+b_j/w_j)/2}
 R^{1-1/(2w_j)}.
\]
Divide by $R^{1-1/(2w_j)}$ and raise to the power $2w_j$ to obtain
\eqref{eq:rich-main}.
\end{proof}

One may take the minimum of \eqref{eq:rich-vinh} and
\eqref{eq:rich-main} over all indices satisfying
\eqref{eq:rich-threshold}.
When $\nu=0$, there are no such indices and
\eqref{eq:rich-vinh} is the conclusion of the main theorem.

\subsubsection*{The standard and dual-standard ranges}

\cref{cor:rich-main} has simpler formulations in the standard and dual-standard ranges of \cref{cor:standard,cor:dual-standard}.
Let $0\leq i<\nu$ with $i\leq\delta-\nu$, and put $j=\nu-i$.
Then $b_j=i\delta$, $w_j=1$, and $\rho+j=n-i$.
Consequently,
\begin{equation}\label{eq:rich-nonexceptional}
 \Delta\geq Cq^{(n-i-1)/2}|P|^{1/2}
 \quad\Longrightarrow\quad
 R_t(P)\ll\frac{q^{i\delta+n-i}|P|}{\Delta^2}.
\end{equation}
For $i=\nu$, the bound on the right is the Vinh bound, since
$\nu\delta+n-\nu=nd$, and holds without the richness threshold.

In particular, suppose $d\geq n$ and
$0\leq i\leq\min\{n,d-n\}$.  Under the richness hypothesis in
\eqref{eq:rich-nonexceptional},
\begin{equation*}
 R_t(P)\ll\frac{q^{id+n-i}|P|}{\Delta^2}.
\end{equation*}
Similarly, if $n\geq d-2$ and
$0\leq i\leq\min\{d-1,n+2-d\}$, the same hypothesis gives
\begin{equation*}
 R_t(P)\ll\frac{q^{(i+1)n}|P|}{\Delta^2}.
\end{equation*}

\subsection{Exceptional orthogonal projections}
\label{sec:exceptional-projections}

Let $G(d,n+d)$ denote the set of $d$-dimensional linear subspaces of
$\Fq^{n+d}$.  For $V\in G(d,n+d)$, define the orthogonal projection by
$\pi_V(x)=x+V^\perp$, where orthogonality is with respect to the standard
dot product; thus $\pi_V$ takes values in $\Fq^{n+d}/V^\perp$ and its
fibers are affine $n$-flats.  For $A\subseteq\Fq^{n+d}$ with $|A|=q^a$
and $0<s<\min\{a,d\}$, write
\[
 \mathcal E_s(A)
 =\{V\in G(d,n+d):|\pi_V(A)|<q^s\}.
\]
The exceptional-set estimates originating with
Chen~\cite[Theorem~1.2]{Chen2018Projections} give
\begin{equation}\label{eq:proj-chen}
 |\mathcal E_s(A)|\ll q^{C(a,s)},
 \qquad C(a,s)=nd+s-\max\{a,d\}.
\end{equation}
Bright and Gan~\cite{BrightGanFinite2023} proved the further estimate
\begin{equation}\label{eq:proj-bright-gan}
 |\mathcal E_s(A)|\ll (\log q)q^{G(a,s)},
 \qquad G(a,s)=\max\{nd+2(s-a),\,n(d-1)\},
\end{equation}
provided $a-2s>n-d$.

These questions originate in Euclidean projection theory.
Marstrand's projection theorem \cite{Marstrand1954}, extended to higher
dimensions by Mattila~\cite{Mattila1975}, states that the orthogonal
projection of a Borel set $E\subseteq\mathbb R^{n+d}$ onto almost every
$d$-dimensional subspace has Hausdorff dimension
$\min\{d,\dim_H E\}$.  Refinements due to Kaufman, Falconer, and
Peres--Schlag~\cite{Kaufman1968,Falconer1982,PeresSchlag2000} bound the
Hausdorff dimension of the exceptional directions, rather than merely
showing that they have measure zero.

Influence runs between the continuous and finite settings in both directions.  Chen's work adapts these
continuous exceptional-set estimates to finite fields.  Conversely,
finite-field results led Lund, Pham, and Thu~\cite{LundPhamThu2025} to
formulate a conjecture for continuous radial projections, subsequently
proved by Bright and Gan~\cite{BrightGanRadial2023} in all dimensions
and by Orponen, Shmerkin, and Wang~\cite{OrponenShmerkinWang2024} in the
plane.

We prove a new bound on the number of exceptional orthogonal projections that follows easily from \cref{thm:main}.
Our new estimate improves earlier bounds when $s$ is much smaller than $a/2$.
The precise (and somewhat complicated) range for which we obtain improvements is discussed after the proof.

\begin{corollary}[Exceptional orthogonal projections]
\label{cor:exceptional-projections}
Fix $0<a<n+d$ and $0<s<\min\{a,d\}$, and let
$A\subseteq\Fq^{n+d}$ satisfy $|A|=q^a$.  Retain the parameters
$\nu,\delta,\rho,b_j,w_j$ of \cref{thm:main}.
For every $1\leq j\leq\nu$ such that
\begin{equation}\label{eq:proj-hypothesis}
 a-2s>\rho+j-1,
\end{equation}
we have, for all sufficiently large $q$,
\begin{equation}\label{eq:proj-main}
 |\mathcal E_s(A)|\ll q^{B_j(a,s)},
 \qquad
 B_j(a,s)=b_j+w_j(\rho+j-a)+(2w_j-1)s.
\end{equation}
\end{corollary}

\begin{proof}
Fix $j$, and abbreviate $b=b_j$ and $w=w_j$.
Put $N=|\mathcal E_s(A)|$, and assume $N>0$.
For each $V\in\mathcal E_s(A)$, take all affine $n$-flats parallel to
$V^\perp$ that meet $A$, and let $\cL$ be the union of these families.
Distinct choices of $V$ yield distinct directions, and the intersections of the flats in each family with $A$ form a partition of $A$.
Consequently,
\[
 R:=|\cL|\leq q^sN,
 \qquad I(A,\cL)=q^aN.
\]
Identify these affine flats with their projective closures in
$\PG(n+d,q)$.  Since $s<d$, for sufficiently large $q$ we have
\begin{equation}\label{eq:proj-deviation}
 I(A,\cL)-q^{-d}|A|R
 \geq (1-q^{s-d})q^aN
 \geq \tfrac12q^aN.
\end{equation}
The exponent function in \cref{thm:main} satisfies
\[
 F_{\nu,\delta}(m)
 \leq
 \max\left\{
  \left(2-\frac1w\right)m+\frac bw+j,\,
  2m+j-1
 \right\}.
\]
For $m\leq b+w$, this follows from
\eqref{eq:rich-affine-extension}.  For $m\geq b+w$, it follows because
$F_{\nu,\delta}(m)-2m$ is nonincreasing and takes the value $j-1$ at
$m=b+w$.  Therefore \cref{thm:main},
\eqref{eq:proj-deviation}, and $R\leq q^sN$ give
\[
 q^aN
 \ll
 q^{(a+\rho+j+b/w)/2}(q^sN)^{1-1/(2w)}
 +q^{(a+\rho+j-1)/2+s}N.
\]
By \eqref{eq:proj-hypothesis}, the second term is $o(q^aN)$ and can be
absorbed into the left-hand side.  Dividing by $q^aN^{1-1/(2w)}$ yields
\[
 N^{1/(2w)}
 \ll q^{(\rho+j+b/w-a)/2+s(1-1/(2w))}.
\]
Raising to the power $2w$ proves \eqref{eq:proj-main}.
\end{proof}

\paragraph{Comparison with earlier bounds.}
Condition~\eqref{eq:proj-hypothesis} implies $a-2s>n-d$, so the
Bright--Gan bound applies throughout the range of
\cref{cor:exceptional-projections}.  If $a-s\geq n$, then
$\mathcal E_s(A)=\varnothing$, since each fiber contains at most $q^n$
points.  We may therefore assume $0<\tau:=a-s<n$.
Put $h=a-2s$ and $u_+=\max\{u,0\}$, and for each admissible $j$ define
\[
 D_j=nd-b_j-w_j(\rho+j)+(w_j-1)h.
\]
The three exponents can then be written as
\[
 \begin{aligned}
 B_j(a,s)&=nd+s-a-D_j,\\
 C(a,s)&=nd+s-a-(d-a)_+,\\
 G(a,s)&=nd+s-a-\min\{\tau,n-\tau\}.
 \end{aligned}
\]
Thus the $j$th estimate gives a strict power saving over both
\eqref{eq:proj-chen} and \eqref{eq:proj-bright-gan} exactly when
\begin{equation*}
 D_j>
 \max\bigl\{(d-a)_+,\,\min\{\tau,n-\tau\}\bigr\}.
\end{equation*}
When $B_j(a,s)=G(a,s)$, it still removes the logarithmic factor from
\eqref{eq:proj-bright-gan}.  One may take the minimum over all
admissible $j$, together with the two earlier bounds.

\subsection{Furstenberg sets}\label{sec:furstenberg}

Furstenberg sets over finite fields generalize Kakeya sets by requiring large intersections with affine flats in prescribed directions. More precisely, an $(n,t,\delta)$-Furstenberg set in $\mathbb F_q^{n+d}$ contains at least $t$ points on some affine $n$-flat in each of at least a $\delta$-fraction of the $n$-dimensional directions. The higher-dimensional Kakeya problem, corresponding to $t=q^n$ and $\delta=1$, was studied by Ellenberg, Oberlin, and Tao~\cite{EOT}, while Ellenberg and Erman~\cite{EE} established general lower bounds for sets containing a rich flat in every direction. Dhar, Dvir, and Lund~\cite{DharDvirLund2021} subsequently gave elementary proofs and improved bounds, including estimates approaching the natural scale $tq^d$ when $t$ is sufficiently large. More recently, Kumar and Mon~\cite[Theorem~2.20]{KM} obtained bounds for a fraction of directions with linear dependence on $\delta$ and an error factor controlled by the codimension and richness rather than the ambient dimension, with applications to linear hashing and list decoding. Here we apply our incidence estimates both directly and through descent to smaller flats, obtaining further improvements in ranges where the fraction of rich directions is small.

Throughout this subsection, $\delta$ denotes the fraction of directions, consistently with its use as a failure probability in \cite{LundTwoSided}; the normalized codimension parameter
in the definition of $\Omega$ is unrelated. All implied constants
depend only on $n,d$. Affine flats are identified with their projective
closures when applying our incidence theorems.

Our first result is a direct application of \cref{thm:main}.
It does not depend on the directions of the rich flats, and instead only uses the fact that there are  many rich flats relative to a Furstenberg set.

\begin{proposition}[Direct incidence bound]\label{prop:furst-direct}
Let $K$ be an $(n,t,\delta)$-Furstenberg set, and put
\[
 m_0=\max\{0,nd+\log_q\delta\},\qquad
 Q_0=q^{2m_0-\Omega_{n,d}(m_0)}.
\]
Then
\begin{equation}\label{eq:furst-direct}
 |K|\gg\min\{tq^d,t^2Q_0\}.
\end{equation}
\end{proposition}

\begin{proof}
Write $s=|K|$. There are $\qbinom{n+d}{n}\geq q^{nd}$ directions,
so we can choose a family $\cL$ of $\lceil\delta q^{nd}\rceil$
witnessing flats. If $s<tq^d/2$, Theorem~\ref{thm:main} gives
\[
 \frac t2|\cL|\leq
 I(K,\cL)-q^{-d}s|\cL|
 \ll s^{1/2}q^{\Omega_{n,d}(\log_q|\cL|)/2}.
\]
Squaring proves the claim. Replacing $|\cL|$ by $q^{m_0}$ costs only
a constant factor, since these sizes differ by at most a factor two
and all slopes of $\Omega$ lie between one and two.
\end{proof}

Our second result uses \cref{thm:main} together with the incidence-based approach to the Furstenberg problem developed in \cite{DharDvirLund2021}.

\begin{proposition}[Incidence descent]\label{prop:furst-descent}
Let $K$ be an $(n,t,\delta)$-Furstenberg set. For each integer
$1\leq a<n$ with $t\geq16q^a$, put
\[
 m_a=\max\{0,(n-a+1)(d+a)+\log_q\delta\},\qquad
 Q_a=q^{2m_a-\Omega_{n-a,d+a}(m_a)-2a}.
\]
Then
\begin{equation}\label{eq:furst-descent}
 |K|\gg\min\{tq^d,t^2Q_a\}.
\end{equation}
\end{proposition}

\begin{proof}
We modify the proof of \cite[Theorem~1.5]{LundTwoSided},
retaining its local Chebyshev estimate and quotient-and-average
counting step, and replacing its global Chebyshev estimate by
Theorem~\ref{thm:main}. In that paper's notation, the ambient dimension
is $n+d$, the output dimension is $b=d$, and the descent dimension is $a$.

Let $\mathcal B$ be the family of affine $(n-a)$-flats containing at
least $t/(2q^a)$ points of $K$. In a witnessing $n$-flat, a random
$(n-a)$-flat fails this condition with probability at most $4q^a/t$.
Consequently, for at least half the contained $(n-a)$-directions,
at least half the parallel translates belong to $\mathcal B$.
Apply the first branch of \cite[Claim~2.1]{LundTwoSided}, with
ambient dimension $n+d$, parent dimension $k=n$, subflat dimension
$\ell=n-a$, flag fraction $\alpha=\delta/2$, and density parameter
$\gamma=1/2$. This gives
\[
 \frac{|\mathcal B|}{q^{d+a}\qbinom{n+d}{n-a}}
 \geq c_a\delta,\qquad
 c_a=\frac12\left(\frac{1}{2(1+q^{1-a})}\right)^{d+a}
 \geq\frac12\,4^{-(d+a)}.
\]
Thus $\mathcal B$ contains a family $\cL$ of size comparable to
$q^{m_a}$, with constants depending only on the dimensions.
If $s=|K|<tq^d/4$, its incidence discrepancy is at least
$t|\cL|/(4q^a)$. Apply Theorem~\ref{thm:main} to $(n-a)$-flats of
codimension $d+a$, square, and absorb the constant factor in the
family size as in the preceding proof.
\end{proof}

\paragraph{Comparison with earlier bounds.}
The classical incidence estimate, Theorem~\ref{thm:vinh}, gives
\begin{equation}\label{eq:furst-classical}
 |K|\gg\min\{tq^d,\delta t^2\}.
\end{equation}
For $\delta=1$ and $n>d$, Dhar--Dvir--Lund
\cite[Theorem~3]{DharDvirLund2021} obtain
$(1-q^{d-n}-\sqrt{q^d/t})tq^d$. Their argument extends to a fraction
$\delta$ of the directions, giving \eqref{eq:furst-classical}; more
precisely, Theorem~\ref{thm:vinh} gives
\[
 |K|\geq\left(1-(1+o(1))\sqrt{\frac{q^d}{\delta t}}\right)tq^d.
\]
Since $\Omega_{n,d}(m_0)\leq m_0+nd$,
\eqref{eq:furst-direct} includes \eqref{eq:furst-classical}, up to constants.

For $n\geq2$ and $t\geq16q$, the master theorem
\cite[Theorem~1.5]{LundTwoSided}, with descent dimension one, gives
\begin{equation}\label{eq:furst-lund}
 |K|\gg\min\{tq^d,\delta t^2q^{d-1}\}.
\end{equation}
Indeed, unless $s\geq tq^d/4$, set
$\tau=tq^d/s-1$ and $E=s/q^{d+1}$ in that theorem. Its density
parameter satisfies
$\gamma=1-8tq/(t-s/q^d)^2\geq1/9$, so its DD branch yields
\eqref{eq:furst-lund}. Threshold equality can be handled by a limit
from below. For fixed dimensions and $t/q\to\infty$, this is the
strongest power of $q$ obtained from that master theorem by choosing
the descent dimension: descent by $a$ gives the coefficient
$\delta q^{d-a}$ in place of $\delta q^{d-1}$.
Similarly, $\Omega_{n-a,d+a}(m_a)\leq m_a+(n-a)(d+a)$ shows that
\eqref{eq:furst-descent} contains the corresponding Chebyshev-descent
bound, since $Q_a\geq\delta q^{d-a}$.

The additional gains are principally for small direction fractions.
For fixed dimensions, \eqref{eq:furst-lund} already reaches the scale
$tq^d$ when $\delta t\gg q$, so an improvement in the power of $q$
requires leaving that range. When $\delta$ is small, the witnessing
families enter the ranges where our incidence estimate improves on
the classical one.

\paragraph{The range $d>2n$.}
The direct bound has a particularly simple form here. For
$1\leq j\leq n$ and $q^{-jd}\leq\delta\leq q^{-(j-1)d}$,
\begin{equation*}
 |K|\gg\min\left\{tq^d,
 t^2\min\{\delta q^{j(d-1)},q^{1-j}\}\right\}.
\end{equation*}
This follows by substituting $m_0\in[(n-j)d,(n-j+1)d]$ into the
regular pieces of $\Omega$. In particular, its quadratic coefficient
is larger than that of \eqref{eq:furst-lund} by an unbounded factor
when $\delta=o(q^{-d})$, until the bound is limited by $tq^d$ or
superseded by the trivial bound $|K|\geq t$.
In this range, substitution in the same formulas shows
$Q_a\leq\max\{Q_0,\delta q^{d-1}\}$ for every $1\leq a<n$.
Thus descent does not add a power improvement beyond the larger of
\eqref{eq:furst-direct} and \eqref{eq:furst-lund} when $d>2n$.

\paragraph{Neither new bound dominates the other in general.}
The following examples exhibit a strict gain for each of our bounds.
We take $q\to\infty$ with dimensions fixed. Every entry includes the
trivial bound $|K|\geq t$; the descent column is optimized over $a$.
The first comparison column uses the $\delta$-extension of the
argument behind \cite[Theorem~3]{DharDvirLund2021}, not that theorem
outside its stated hypotheses.
\begin{center}
\renewcommand{\arraystretch}{1.15}
\begin{tabular}{c|c|c|c|c}
$(n,d,t,\delta)$ & Classical & Lund master & Direct & Descent\\ \hline
$(3,7,q^3,q^{-9})$ & $q^3$ & $q^3$ & $q^5$ & $q^4$\\
$(5,3,q^5,q^{-9})$ & $q^5$ & $q^5$ & $q^6$ & $q^7$
\end{tabular}
\end{center}
For the first row,
$\Omega_{3,7}(12)=25$ and $\Omega_{2,8}(15)=30$, giving
$Q_0=q^{-1}$ and $Q_1=q^{-2}$.
For the second,
$\Omega_{5,3}(6)=16$ and $\Omega_{4,4}(11)=23$, giving
$Q_0=q^{-4}$ and $Q_1=q^{-3}$.
Thus the direct argument and incidence descent should both be retained,
together with \eqref{eq:furst-lund} and $|K|\geq t$.

\section{A graph-theoretic lemma}\label{sec:graph-lemma}

For a graph $G$ and a vertex $v\in V(G)$, denote the open neighborhood of $v$ by $N(v)$ and its degree by $d(v)=|N(v)|$. For $L,R\subseteq V(G)$, write
\[
 E(L,R)=\#\{(u,v)\in L\times R:u\sim v\}.
\]
The next lemma follows from an observation of Thomason that several pseudorandom properties of graphs can be established by bounding the number of copies of $C_4$; see \cite[Theorem~2]{thomason1989dense} and \cite[Theorem~11]{KohLundXuYoo} for related statements.

\begin{lemma}\label{lem:graph}
Let $G$ be a bipartite graph with vertex classes $L$ and $U$, and let $p\geq0$ satisfy
\[
 p\leq |L|^{-1}|U|^{-1}E(L,U).
\]
Then, for every $R\subseteq U$,
\begin{equation*}
 \bigl(E(L,R)-p|L||R|\bigr)^2
 \leq |R|
 \sum_{\substack{v_1,v_2\in L\\|N(v_1)\cap N(v_2)|>p^2|U|}}
 |N(v_1)\cap N(v_2)|.
\end{equation*}
\end{lemma}

\begin{proof}
Since $\sum_{u\in U}d(u)=E(L,U)\geq p|L||U|$, Cauchy--Schwarz gives
\begin{align*}
 \bigl(E(L,R)-p|L||R|\bigr)^2
 &=\left(\sum_{u\in R}(d(u)-p|L|)\right)^2\\
 &\leq |R|\sum_{u\in R}(d(u)-p|L|)^2\\
 &\leq |R|\sum_{u\in U}(d(u)-p|L|)^2\\
 &\leq |R|\left(\sum_{u\in U}d(u)^2-p^2|L|^2|U|\right)\\
 &=|R|\sum_{v_1,v_2\in L}
 \bigl(|N(v_1)\cap N(v_2)|-p^2|U|\bigr)\\
 &\leq |R|
 \sum_{\substack{v_1,v_2\in L\\|N(v_1)\cap N(v_2)|>p^2|U|}}
 |N(v_1)\cap N(v_2)|.
\end{align*}
\end{proof}

The lemma will be used twice: first for the point--flat incidence graph and then, in a zero-density form, for the graph that records a prescribed intersection dimension between two flats.

\section{New proofs of \texorpdfstring{\cref{thm:vinh,thm:kong-tamo}}{Theorems 1 and 2}}
\label{sec:classical-proofs}
In this section, we show how to use \cref{lem:graph} in place of spectral arguments to prove \cref{thm:vinh,thm:kong-tamo}.

\begin{proof}[Proof of \cref{thm:vinh}]
Put

$$
 \theta_j=\frac{q^{j+1}-1}{q-1},
 \qquad
 p=\frac{\theta_n}{\theta_{n+d}},
$$
and let $\cG$ denote the set of all $n$-flats in
$\PG(n+d,q)$.  Apply \cref{lem:graph} to the incidence graph with
left class $P$, right class $\cG$, and $R=\cL$.

Every point lies on
$$
 r=\qbinom{n+d}{n}
$$
$n$-flats, while two distinct points lie on
$$
 \lambda=\qbinom{n+d-1}{n-1}
$$
$n$-flats.  The incidence structure is a $2$-design, so
$$
 \frac{r}{|\cG|}=p,
 \qquad
 \lambda
 =r\frac{\theta_n-1}{\theta_{n+d}-1}
 <r\frac{\theta_n}{\theta_{n+d}}
 =p^2|\cG|.
$$

Thus only the diagonal pairs of points occur in the sum in
\cref{lem:graph}, and hence
$$
 \bigl|I(P,\cL)-p|P||\cL|\bigr|^2
 \leq r|P||\cL|.
$$
Since
$$
 r=(1+O(q^{-1}))q^{dn},
$$
we obtain
$$
 \bigl|I(P,\cL)-p|P||\cL|\bigr|
 \precsim q^{dn/2}|P|^{1/2}|\cL|^{1/2}.
$$

Finally,
$$
 p=q^{-d}+O(q^{-d-n-1}),
$$
and the resulting change in the main term is
$o\bigl(q^{dn/2}|P|^{1/2}|\cL|^{1/2}\bigr)$, using the trivial
bounds $|P|\leq\theta_{n+d}$ and
$|\cL|\leq\qbinom{n+d+1}{n+1}$.  This proves the theorem.
\end{proof}

\begin{proof}[Proof of \cref{thm:kong-tamo}]
Use the same notation $p=\theta_n/\theta_{n+d}$, but now apply
\cref{lem:graph} to the point--flat incidence graph with left class
$\cL$, right class the full point set of $\PG(n+d,q)$, and $R=P$.
A flat has $\theta_n$ common point-neighbors with itself, while two
distinct $n$-flats have at most $\theta_{n-1}$ common point-neighbors.
Therefore,
$$
 \bigl|I(P,\cL)-p|P||\cL|\bigr|^2
 \leq
 |P|\bigl(\theta_n|\cL|
 +\theta_{n-1}|\cL|^2\bigr).
$$

Since
$$
 \theta_n\asymp q^n,
 \qquad
 \frac{\theta_{n-1}}{\theta_n}<q^{-1},
$$
it follows that
$$
 \bigl|I(P,\cL)-p|P||\cL|\bigr|
 \ll
 q^{n/2}|P|^{1/2}|\cL|^{1/2}
 \bigl(1+q^{-1}|\cL|\bigr)^{1/2}.
$$

As in the preceding proof, replacing $p$ by $q^{-d}$ changes the
left-hand side by a term absorbed by the displayed bound.  This proves
the theorem.
\end{proof}

\section{Initial bounds}\label{sec:initial-bounds}

We now convert the incidence problem into estimates for ordered pairs of flats with a fixed intersection dimension. The first proposition controls the incidence error by these pair counts; the remaining results prepare the recursive inequalities that will be solved in \cref{sec:main-proof}.

Put
\[
 \theta_t=\frac{q^{t+1}-1}{q-1},
 \qquad
 p=\frac{\theta_n}{\theta_{n+d}}.
\]
For $-1\leq k\leq n$, let
\[
 N_k(\cL)=\#\{(\ell_1,\ell_2)\in\cL^2:\dim(\ell_1\cap\ell_2)=k\},
\]
where $\dim(\varnothing)=-1$.

\begin{proposition}[The first Cauchy--Schwarz step]\label{prop:first-cs}
Let $P$ be a set of points in $\PG(n+d,q)$ and let $\cL$ be a set of $n$-flats. Then
\begin{equation}\label{eq:first-cs}
 \bigl|I(P,\cL)-p|P||\cL|\bigr|^2
 \ll |P|
 \sum_{k=\max\{0,n-d+1\}}^n q^k N_k(\cL).
\end{equation}
In particular, the diagonal term $k=n$ contributes $q^{m+n}$ when $|\cL|=M=q^m$.
\end{proposition}

\begin{proof}
Let $\cG$ be the set of all $n$-flats and let $U$ be the point set of $\PG(n+d,q)$. In the point--flat incidence graph with vertex classes $\cG$ and $U$, every $n$-flat contains $\theta_n$ points, while $|U|=\theta_{n+d}$. Hence the edge density between any subfamily $\cL\subseteq\cG$ and the full point set $U$ is exactly
\[
 \frac{E(\cL,U)}{|\cL||U|}
 =\frac{\theta_n}{\theta_{n+d}}=p.
\]
We may therefore apply \cref{lem:graph} with left class $\cL$ and right subset $P$.

If two flats $\ell_1,\ell_2\in\cL$ meet in a $k$-flat, then their common point-neighborhood has size $\theta_k$; if they are disjoint, the common neighborhood is empty. The threshold in \cref{lem:graph} is
\[
 p^2|U|=\frac{\theta_n^2}{\theta_{n+d}}.
\]
If $n\geq d$, the exact identity
\[
 \theta_n^2-\theta_{n-d}\theta_{n+d}
 =\frac{q^{n-d+1}(q^d-1)^2}{(q-1)^2}>0
\]
shows that $\theta_{n-d}<p^2|U|$. Since $\theta_k$ increases with
$k$, no layer with $k\leq n-d$ contributes. If $n<d$, every
nonempty intersection already has $k\geq0>n-d$.
Thus in both cases the contributing nonempty layers satisfy
\[
 k\geq\max\{0,n-d+1\}.
\]
Grouping the ordered pairs $(\ell_1,\ell_2)$ according to their intersection dimension, \cref{lem:graph} gives
\[
 \bigl|I(P,\cL)-p|P||\cL|\bigr|^2
 \leq |P|
 \sum_{k=\max\{0,n-d+1\}}^n\theta_k N_k(\cL).
\]
Finally, $\theta_k\ll q^k$ for fixed $k$, which proves \eqref{eq:first-cs}. The identity $N_n(\cL)=|\cL|=M=q^m$ gives the stated diagonal contribution.
\end{proof}

For $-1\leq k\leq n$, let $N_{k,n,d}(M)$ be the maximum of $N_k(\cL)$ over all families of $M$ $n$-flats in $\PG(n+d,q)$.

\begin{lemma}[Projective duality]\label{lem:duality}
For all admissible $k,n,d,M$,
\[
 N_{k,n,d}(M)=N_{k+d-n-1,\,d-1,\,n+1}(M).
\]
\end{lemma}

\begin{proof}
Projective duality in $\PG(n+d,q)$ sends $n$-flats to $(d-1)$-flats. If $\cL$ is a family of $M$ $n$-flats, then
\[
 \cL^*=\{\ell^*:\ell\in\cL\}
\]
is a family of $M$ $(d-1)$-flats in
\[
 \PG(n+d,q)=\PG\bigl((d-1)+(n+1),q\bigr).
\]
If $\dim(\ell_1\cap\ell_2)=k$, then $\dim\langle\ell_1,\ell_2\rangle=2n-k$. Duality reverses joins and intersections, so
\[
 \ell_1^*\cap\ell_2^*=\langle\ell_1,\ell_2\rangle^*
\]
and therefore
\[
 \dim(\ell_1^*\cap\ell_2^*)=(n+d)-(2n-k)-1=k+d-n-1.
\]
This gives the claimed bijection between ordered pairs.
\end{proof}

We next work in the normalized model of $\nu$-flats in $\PG(\nu+\delta,q)$, where $\delta>\nu$. Given a family $\cL$ of $M=q^m$ $n$-flats in $\PG(n+d,q)$, define the normalized family
\begin{equation*}
 \widetilde{\cL}
 =
 \begin{cases}
  \cL,&d>n,\\
  \cL^*,&d\leq n.
 \end{cases}
 \qquad
 |\widetilde{\cL}|=|\cL|=M=q^m.
\end{equation*}
Thus $\widetilde{\cL}$ is always a family of $\nu$-flats in $\PG(\nu+\delta,q)$. Define
\[
 N_h(\widetilde{\cL})
 =\#\{(\ell_1,\ell_2)\in\widetilde{\cL}^{\,2}:\dim(\ell_1\cap\ell_2)=h\},
 \qquad -1\leq h\leq \nu.
\]
In particular, $N_\nu(\widetilde{\cL})=M$.

Fix two $\nu$-flats $\ell_1,\ell_2$ with $\dim(\ell_1\cap\ell_2)=k$, and let
\[
 J_{\nu,\delta}(k,h)
 =\#\{\ell:\dim(\ell\cap\ell_1)=\dim(\ell\cap\ell_2)=h\},
\]
where $\ell$ ranges over all $\nu$-flats in $\PG(\nu+\delta,q)$.
Note that $J_{\nu,\delta}(k,h)$ does not depend on the specific choice of $\ell_1,\ell_2$.

\begin{proposition}[The second Cauchy--Schwarz step]\label{prop:second-cs}
For every $0\leq h<\nu$,
\begin{equation}\label{eq:second-cs}
 N_h(\widetilde{\cL})^2
 \leq M\sum_{k=-1}^\nu N_k(\widetilde{\cL})J_{\nu,\delta}(k,h).
\end{equation}
\end{proposition}

\begin{proof}
Let $\cG$ be the set of all $\nu$-flats in $\PG(\nu+\delta,q)$. For $\ell\in\cG$, define
\[
 d_h(\ell)=\#\{\ell_1\in\widetilde{\cL}:\dim(\ell_1\cap\ell)=h\}.
\]
Since the right-hand member of each ordered pair counted by $N_h(\widetilde{\cL})$ belongs to $\widetilde{\cL}$, we have
\[
 N_h(\widetilde{\cL})=\sum_{\ell\in\widetilde{\cL}}d_h(\ell).
\]
\cref{lem:graph} applied with $p=0$ gives
\begin{align*}
 N_h(\widetilde{\cL})^2
 &\leq M\sum_{\ell\in\widetilde{\cL}}d_h(\ell)^2
 \leq M\sum_{\ell\in\cG}d_h(\ell)^2.
\end{align*}
Expanding the square and interchanging the order of summation, we obtain
\begin{align*}
 \sum_{\ell\in\cG}d_h(\ell)^2
 &=\sum_{\ell_1,\ell_2\in\widetilde{\cL}}
   \#\{\ell\in\cG:\dim(\ell\cap\ell_1)=\dim(\ell\cap\ell_2)=h\}\\
 &=\sum_{k=-1}^{\nu}N_k(\widetilde{\cL})J_{\nu,\delta}(k,h).
\end{align*}
Substitution into the preceding inequality proves \eqref{eq:second-cs}.
\end{proof}

Recall that $\ell_1,\ell_2$ are $\nu$-flats with $\dim(\ell_1\cap\ell_2)=k$.
For fixed $k$ and $h$, if a third $\nu$-flat meets both $\ell_1$ and $\ell_2$ in $h$-flats and the triple intersection has dimension $\sigma$, then the feasible values for $\sigma$ are
\begin{equation}\label{eq:S-def}
 \cS(k,h)=\bigl\{\sigma\in\mathbb Z:
 \max(-1,k+h-\nu,2h-\nu)\leq\sigma\leq\min(k,h)\bigr\}.
\end{equation}
Define
\begin{equation}\label{eq:e-def}
 e_{\nu,\delta}(k,h,\sigma)
 = (\sigma+1)(k-\sigma)
 +2(h-\sigma)(\nu-h)
 +\delta(\nu-2h+\sigma)
\end{equation}
and
\begin{equation}\label{eq:kappa-def}
 \kappa_{kh}=\max_{\sigma\in\cS(k,h)}e_{\nu,\delta}(k,h,\sigma),
\end{equation}
with the convention that $(k,h)$ is absent when $\cS(k,h)=\varnothing$.

For fixed \(h\), \(J_{\nu,\delta}(k,h)\) is the common-neighbor kernel of the \(h\)-intersection graph: it is the number of common \(h\)-neighbors of two \(\nu\)-flats whose intersection has dimension \(k\).

\begin{proposition}[The exact common-neighbor kernel]\label{prop:common-neighbor}
For every feasible pair $(k,h)$,
\[
 J_{\nu,\delta}(k,h)\ll q^{\kappa_{kh}}.
\]
More precisely, for fixed $\nu$ and $\delta$,
\[
 J_{\nu,\delta}(k,h)=(1+o(1))
 \sum_{\sigma\in\cS(k,h)}q^{e_{\nu,\delta}(k,h,\sigma)}.
\]
\end{proposition}

\begin{proof}
Let $V$ be the $(\nu+\delta+1)$-dimensional vector space underlying $\PG(\nu+\delta,q)$, and let $U_1,U_2\leq V$ be the $(\nu+1)$-dimensional vector subspaces corresponding to $\ell_1,\ell_2$. Thus
\[
 \dim(U_1\cap U_2)=k+1.
\]
For a third $\nu$-flat $\ell$, represented by a $(\nu+1)$-dimensional subspace $W\leq V$, put
\[
 A=W\cap U_1\cap U_2,
 \qquad \dim A=\sigma+1.
\]
We count such subspaces $W$ for each fixed feasible value of $\sigma$.

First choose $A$ inside $U_1\cap U_2$. The number of choices is
\[
 \qbinom{k+1}{\sigma+1}
 =(1+o(1))q^{(\sigma+1)(k-\sigma)}.
\]
Next put $H_i=W\cap U_i$ for $i=1,2$. Each $H_i$ has vector dimension $h+1$, contains $A$, and satisfies
\[
 H_i\cap(U_1\cap U_2)=A.
\]
After quotienting by $A$, one must choose an $(h-\sigma)$-dimensional subspace of $U_i/A$ disjoint from $(U_1\cap U_2)/A$. The standard formula for subspaces disjoint from a fixed subspace gives
\[
 q^{(h-\sigma)(k-\sigma)}
 \qbinom{\nu-k}{h-\sigma}
 =(1+o(1))q^{(h-\sigma)(\nu-h)}
\]
choices for each $H_i$. The two choices therefore contribute
\[
 (1+o(1))q^{2(h-\sigma)(\nu-h)}.
\]
The condition $H_i\cap(U_1\cap U_2)=A$ also implies $H_1\cap H_2=A$, so
\[
 \dim(H_1+H_2)=2h-\sigma+1.
\]
It remains to extend $H_1+H_2$ to a $(\nu+1)$-dimensional subspace $W$ without creating any additional intersection with either $U_1$ or $U_2$. In the quotient $V/(H_1+H_2)$, this amounts to choosing a subspace of dimension
\[
 \nu-2h+\sigma.
\]
Set
\[
 t=\nu-2h+\sigma
 \qquad\text{and}\qquad
 \widetilde V=V/(H_1+H_2).
\]
Then $\dim\widetilde V=\delta+t$. For $i=1,2$, let
\[
 A_i=(U_i+H_1+H_2)/(H_1+H_2)\leq\widetilde V.
\]
Because $(H_1+H_2)\cap U_i=H_i$, one has $\dim A_i=\nu-h$. The desired extension $W$ is equivalent to a $t$-dimensional subspace $X=W/(H_1+H_2)$ of $\widetilde V$ satisfying
\[
 X\cap A_1=X\cap A_2=\{0\}.
\]
If $t=0$, the unique choice $X=\{0\}$ satisfies these conditions
and contributes exactly $1=q^{\delta t}$ extension. Suppose now
that $t\geq1$.
The total number of $t$-dimensional subspaces of $\widetilde V$ is
\[
 \qbinom{\delta+t}{t}=(1+o(1))q^{\delta t}.
\]
For a fixed one-dimensional subspace $L\leq A_i$, the number of $t$-subspaces containing $L$ is
\[
 \qbinom{\delta+t-1}{t-1}\ll q^{(t-1)\delta}.
\]
There are $O_{\nu,\delta}(q^{\nu-h-1})$ possible choices for $L$. Since $\delta>\nu\geq \nu-h$, the number of $t$-subspaces meeting $A_i$ nontrivially is therefore
\[
 O_{\nu,\delta}\bigl(q^{\nu-h-1+(t-1)\delta}\bigr)=o(q^{\delta t}).
\]
The same estimate applies to $A_2$, so a union bound gives the
number of admissible $X$. Including the case $t=0$, this number is
\[
 (1+o(1))q^{\delta t}
 =(1+o(1))q^{\delta(\nu-2h+\sigma)}.
\]
Multiplying the three contributions gives
\[
 (1+o(1))q^{e_{\nu,\delta}(k,h,\sigma)}.
\]
The inequalities defining $\cS(k,h)$ express exactly the requirements that all dimensions in the preceding construction be nonnegative and that the indicated subspaces exist. Summing over the finitely many feasible values of $\sigma$ proves the asymptotic formula. The upper bound by $q^{\kappa_{kh}}$ follows immediately because the number of summands depends only on $\nu$ and $\delta$.
\end{proof}

\begin{lemma}[A useful form of the kernel]\label{lem:kernel-form}
Let $r=2\nu-\delta$. For every feasible pair $(k,h)$,
\begin{equation}\label{eq:kernel-form}
 \kappa_{kh}=\delta(\nu-h)+(h+1)(k-h)+\gamma_{kh},
\end{equation}
where
\begin{equation}\label{eq:gamma-def}
 \gamma_{kh}
 =\max_{\substack{u\in\mathbb Z\\
 \max(0,h-k)\leq u\leq\min(h+1,\nu-h,\nu-k)}}
 u(r-k+1-u).
\end{equation}
\end{lemma}

\begin{proof}
Set $u=h-\sigma$. The inequalities in \eqref{eq:S-def} become
\[
 \max(0,h-k)\leq u\leq\min(h+1,\nu-h,\nu-k).
\]
Substituting $\sigma=h-u$ into \eqref{eq:e-def} and collecting the terms independent of $u$ gives
\[
 e_{\nu,\delta}(k,h,h-u)
 =\delta(\nu-h)+(h+1)(k-h)+u(2\nu-\delta-k+1-u),
\]
which is \eqref{eq:kernel-form}--\eqref{eq:gamma-def}.
\end{proof}

If $h \geq \max\{0,2\nu - \delta\}$, then $e_{\nu,\delta}(k,h,\sigma)$ is maximized for $\sigma=\min(k,h)$; this leads to the following particularly simple form of $J_{\nu,\delta}(k,h)$.

\begin{corollary}[The nonexceptional kernel]\label{cor:nonexceptional-kernel}
If $h\geq\max\{0,2\nu-\delta\}$, then
\[
 J_{\nu,\delta}(k,h)\ll
 \begin{cases}
 q^{(k-h)(h+1)+\delta(\nu-h)},&k>h,\\
 q^{\delta(\nu-h)},&k=h,\\
 q^{2(\nu-h)(h-k)+\delta(\nu-2h+k)},&k<h.
 \end{cases}
\]
\end{corollary}

\begin{proof}
For consecutive feasible values of $\sigma$,
\begin{align*}
 &e_{\nu,\delta}(k,h,\sigma+1)-e_{\nu,\delta}(k,h,\sigma)\\
 &\hspace{2cm}=k+2h+\delta-2\nu-2\sigma-2.
\end{align*}
The hypothesis $h\geq2\nu-\delta$ implies
\[
 k+2h+\delta-2\nu-2\sigma-2
 \geq k+h-2\sigma-2.
\]
If $\sigma<\min(k,h)$, the right-hand side is nonnegative. Hence the maximum is attained at $\sigma=\min(k,h)$, although the preceding value may tie with it in a boundary case. Substituting $\sigma=h$ when $k\geq h$, and $\sigma=k$ when $k<h$, gives the three displayed exponents. The possible tie only changes the implied constant, which is why the conclusion is stated with $\ll$.
\end{proof}

% Replace the whole of Section 7 by this snippet.
% No additions to the current preamble are required.

\section{Bounds for intersecting pairs of flats}\label{sec:main-proof}

We first obtain upper bounds for the extremal pair counts
$N_{k,n,d}(M)$ defined in the preceding section.  These bounds do not
involve a point set.
The main theorem follows by inserting these pair-counting bounds into
 \cref{prop:first-cs}.

\subsection{Coordinatewise fixed-point bounds}

Work first in normalized dimensions $\delta>\nu\geq1$, and write
$M=q^m$. The coordinates below describe upper bounds for the
weighted pair counts $q^hN_h(\widetilde{\cL})$: a coordinate $x_h$
represents a candidate exponent in
$q^hN_h(\widetilde{\cL})\ll q^{x_h}$. The weight $q^h$ is the
contribution of an $h$-dimensional intersection to the squared
incidence error in \cref{prop:first-cs}. Thus the fixed-point
coordinate $\beta_h$ will bound the exponent of this weighted
quantity, rather than that of $N_h$ alone.

For $x\in\mathbb R^\nu$, extend its coordinates by
\begin{equation*}
 \overline x_{-1}=2m-1,
 \qquad \overline x_\nu=m+\nu,
 \qquad \overline x_k=x_k\quad(0\leq k<\nu).
\end{equation*}
The boundary value $2m-1$ comes from
$q^{-1}N_{-1}\leq q^{-1}M^2$, while $m+\nu$ is the exponent of
the exact diagonal contribution $q^\nu N_\nu=q^\nu M$.
Likewise, $N_h\leq M^2$ gives the cap $2m+h$ for each internal
coordinate. The second Cauchy--Schwarz step improves these caps
through the common-neighbor exponents $\kappa_{kh}$.
Define $T_m:\mathbb R^\nu\to\mathbb R^\nu$ by
\begin{equation*}
 (T_mx)_h=
 \min\left\{2m+h,\;
 \frac12\max_{\substack{-1\leq k\leq\nu\\
                        \cS(k,h)\neq\varnothing}}
 \bigl(m+\overline x_k+\kappa_{kh}+2h-k\bigr)\right\},
 \qquad 0\leq h<\nu.
\end{equation*}
Here $\kappa_{kh}$ and the feasible sets $\cS(k,h)$ are as in
\eqref{eq:kappa-def} and \eqref{eq:S-def}.

\begin{theorem}[Fixed-point bounds for intersecting pairs]
\label{thm:pair-count-fixed}
The map $T_m$ is monotone and $1/2$-contractive in the
$\ell^\infty$ norm.  Let
\[
 \beta^{\nu,\delta}(m)
 =(\beta_0^{\nu,\delta}(m),\ldots,\beta_{\nu-1}^{\nu,\delta}(m))
\]
be its unique fixed point. When the dimensions are fixed, abbreviate
$\beta_h^{\nu,\delta}(m)$ to $\beta_h(m)$, or to $\beta_h$ when
$m$ is also understood. For every admissible integer $M=q^m$,
\begin{equation}\label{eq:pair-count-fixed-bound}
 N_{h,\nu,\delta}(M)\ll q^{\beta_h(m)-h},
 \qquad 0\leq h<\nu.
\end{equation}
The implied constant is independent of $M$ and $q$.
\end{theorem}

\begin{proof}
Monotonicity is immediate, and
\[
 \|T_mx-T_my\|_\infty\leq\tfrac12\|x-y\|_\infty,
\]
so the Banach fixed-point theorem gives a unique fixed point.
Put
\[
 Q_h=q^{\beta_h-h}\quad(0\leq h<\nu),
 \qquad Q_{-1}=M^2,\qquad Q_\nu=M.
\]
The fixed-point equation becomes
\begin{equation}\label{eq:pair-Q-fixed}
 Q_h^2=
 \min\left\{M^4,\;
 M\max_{\substack{-1\leq k\leq\nu\\
                 \cS(k,h)\neq\varnothing}}
 q^{\kappa_{kh}}Q_k\right\}.
\end{equation}
For a family $\widetilde{\cL}$ of $M$ normalized flats, abbreviate
$N_h(\widetilde{\cL})$ to $N_h$.  By
\cref{prop:second-cs,prop:common-neighbor},
\begin{equation*}
 N_h^2\ll M\sum_{\substack{-1\leq k\leq\nu\\
                          \cS(k,h)\neq\varnothing}}
 q^{\kappa_{kh}}N_k,
 \qquad N_h\leq M^2.
\end{equation*}
Let
\[
 R=\max\left\{1,\max_{0\leq h<\nu}\frac{N_h}{Q_h}\right\}.
\]
If $R>1$, choose an index $h$ attaining the maximum.  Then
$Q_h<M^2$, so the second branch in \eqref{eq:pair-Q-fixed} is active.
Since $N_k\leq RQ_k$ also at the boundary indices,
\[
 R^2Q_h^2=N_h^2
 \ll RM\sum_k q^{\kappa_{kh}}Q_k
 \ll RQ_h^2.
\]
There are at most $\nu+2$ summands, and hence $R\ll1$.  Taking the
maximum over all families proves \eqref{eq:pair-count-fixed-bound}.
\end{proof}

Returning to the original dimensions, \cref{lem:duality} gives
\begin{equation}\label{eq:pair-count-original}
N_{k,n,d}(M)
\ll
\begin{cases}
q^{\beta_k^{n,d}(m)-k},
& d>n,\quad 0\leq k<n, \\
q^{\beta_{k+d-n-1}^{d-1,n+1}(m)-(k+d-n-1)},
& d\leq n,\quad n-d+1\leq k<n.
\end{cases}
\end{equation}
The remaining layers satisfy
$$
 N_{n,n,d}(M)=M,
$$
together with
 \begin{align*}
 N_{-1,n,d}(M)&\leq M(M-1), \quad d>n,\\
 N_{n-d,n,d}(M)&\leq M(M-1), \quad d\leq n,
 \end{align*}
and
$$
 N_{k,n,d}(M)=0
 \qquad\text{for }k<n-d
$$
when $d\leq n$.

When $d=1$, there are no internal coordinates: any two distinct
$n$-flats in $\PG(n+1,q)$ meet in an $(n-1)$-flat, so
$$
 N_{n-1,n,1}(M)=M(M-1),
 \qquad
 N_{n,n,1}(M)=M,
$$
and all other pair counts vanish.  We assume $d\geq2$ until the
application to incidences below.

\subsection{Coefficient bounds for the supersolution}

Our eventual goal is to give explicit upper bounds for the fixed point
$\beta^{\nu,\delta}(m)$.  In the next subsection, for each range of
$m$, we will do this by constructing a simple vector $C$ satisfying

$$
 T_m(C)\leq C.
$$
The comparison principle will then give $\beta\leq C$, and the largest
coordinate of $C$ will recover the piecewise-linear exponent
$F_{\nu,\delta}(m)$.

The purpose of the present subsection is only to prepare the algebra
needed to verify the inequality $T_m(C)\leq C$.  Each coordinate of
$T_m$ is a maximum over intersection layers $k$, with coefficients
determined by $\kappa_{kh}$.  We show that these coefficients satisfy
two uniform estimates relative to the breakpoints $b_j$: one for
layers $k\geq j$ and one for layers $k<j$.  We also identify the
largest coefficient coming from a lower layer.  These estimates explain
the role of the breakpoints $b_j$ and widths $w_j$, and will allow all
branches of $T_m$ to be checked simultaneously in the next subsection.

Recall the definition of $\kappa_{kj}$ in \cref{eq:kappa-def}. 
Put
\begin{equation*}
 \Lambda_j=\kappa_{j-1,j},
 \qquad f_j(t)=t(r-j+1-t),
 \qquad 0\leq j<\nu.
\end{equation*}
In particular, $\Lambda_0=\kappa_{-1,0}$.

\begin{lemma}[Diagonal coefficients, adjacent layers, and breakpoints]
\label{lem:breakpoint-coefficients}
For $0\leq j<\nu$,
\begin{equation}\label{eq:diagonal-adjacent}
 \kappa_{jj}=b_j,
 \qquad
 \Lambda_j=
 \max_{\substack{-1\leq k<j\\\cS(k,j)\neq\varnothing}}
 \kappa_{kj}\leq b_j.
\end{equation}
For $1\leq j<\nu$, equality $\Lambda_j=b_j$ holds exactly when
$w_j=2$; otherwise $\Lambda_j\leq b_j-1$.  Moreover,
\begin{equation}\label{eq:adjacent-zero}
 \Lambda_0=\delta(\nu-1)+2\nu=2b_0-b_\ast,
\end{equation}
and the breakpoints satisfy
\begin{equation}\label{eq:threshold-order}
 b_j+w_j\leq b_{j-1}\quad(1\leq j\leq\nu),
 \qquad b_0\leq b_\ast\leq(\nu+1)\delta.
\end{equation}
\end{lemma}

\begin{proof}
By \cref{lem:kernel-form},
\[
 \kappa_{jj}=\delta(\nu-j)+
 \max_{0\leq t\leq\min\{j+1,\nu-j\}}f_j(t)=b_j.
\]
Indeed, $r-j+1\leq\nu-j$, so the additional constraint
$t\leq\nu-j$ cannot exclude every maximizer of $c_j$.
All maxima here and below are over integers.

For a feasible lower layer $k=j-v$, where $v\geq1$, the same formula
gives
\begin{equation}\label{eq:lower-layer-max-form}
 \kappa_{j-v,j}=\delta(\nu-j)+
 \max_{v\leq t\leq\min\{j+1,\nu-j\}}
 \bigl\{f_j(t)-v(j+1-t)\bigr\}.
\end{equation}
Increasing $v$ shrinks the interval and increases the nonnegative
penalty.  Thus the largest lower-layer coefficient occurs at $v=1$
and is at most $b_j$.  Equality requires $t=j+1$ to maximize $f_j$.
Since
\[
 f_j(j+1)-f_j(j)=r-3j,
\]
this occurs exactly when $r\geq3j$.  In that case $t=j+1$ is
feasible, since $\nu\geq r+1\geq3j+1$.  Otherwise integrality gives
a loss of at least one.  At $j=0$, \eqref{eq:lower-layer-max-form}
has only $t=1$, giving \eqref{eq:adjacent-zero}, because
$c_0=\max\{0,r\}$.

For the ordering, $b_\nu=0$, $w_\nu=1$, and
$b_{\nu-1}=\delta$.  If $1\leq j<\nu$ and $r<3j$, a maximizer
for $c_j$ can be chosen with $t\leq j$; using
$f_{j-1}(t)=f_j(t)+t$ gives $c_{j-1}\geq c_j$.  If $r\geq3j$,
both maxima are attained at their right endpoints.  Consequently,
\[
 b_{j-1}-b_j\geq\delta\quad(r<3j),\qquad
 b_{j-1}-b_j=2(\delta-\nu)+4j\quad(r\geq3j),
\]
which proves the first part of \eqref{eq:threshold-order}.
Finally, $b_0=\nu\delta+\max\{0,r\}$ and
$b_\ast=b_0+\max\{0,-r\}$ give the terminal ordering.
\end{proof}

\begin{lemma}[Coefficient comparisons]\label{lem:coefficient-comparisons}
Let $0\leq j\leq h<\nu$.  For every feasible pair $(k,h)$,
\begin{align}
 \kappa_{kh}+(j+1)(2h-k-j)&\leq b_j,
 &&j\leq k\leq\nu,\label{eq:coefficient-upper}\\
 \kappa_{kh}+2(j+1)(h-j)&\leq\Lambda_j,
 &&-1\leq k<j.\label{eq:coefficient-lower}
\end{align}
\end{lemma}

\begin{proof}
For \eqref{eq:coefficient-upper}, first let $k<\nu$ and write
$h=j+u$, $k=j+v$.  Fix a feasible integer $t$ in
\eqref{eq:gamma-def}, and put $x=t-u$.  After using
\cref{lem:kernel-form}, the required nonnegative difference is
\begin{align*}
 &\delta u+c_j-
 \bigl\{t(r-j-v+1-t)+u(v+j-u+1)\bigr\}\\
 &\qquad=c_j-f_j(x)+vx+2u(\delta-\nu+x+u-1).
\end{align*}
If $x\geq0$, feasibility gives $x\leq j+1$, and every term on the
last line is nonnegative.  If $x=-y<0$, then $0<y\leq u$ and
$t\leq\nu-k$ gives $v\leq\nu-j-u+y$.  Substitution bounds the
last line below by
\[
 c_j+(2u-y)(\delta-\nu+u-1)\geq0.
\]
Taking the maximum over $t$ proves the inequality for $k<\nu$.
For $k=\nu$, the only feasible value of $\sigma$ is $h$, so
\[
 \kappa_{\nu h}+(j+1)(2h-\nu-j)
 =\delta(\nu-j)-u(\delta-\nu+u-1)\leq b_j.
\]

For \eqref{eq:coefficient-lower}, write $h=j+u$ and $k<j$.
Every $\sigma\in\cS(k,j+u)$ also belongs to $\cS(k,j)$, and
\begin{align*}
 e_{\nu,\delta}(k,j+u,\sigma)+2(j+1)u
 &=e_{\nu,\delta}(k,j,\sigma)\\
 &\quad-2u(\delta-\nu+j-\sigma+u-1)\\
 &\leq e_{\nu,\delta}(k,j,\sigma).
\end{align*}
Here $\sigma\leq k<j$ and $\delta>\nu$.  Maximizing over $\sigma$
and using \eqref{eq:diagonal-adjacent} proves the claim.
\end{proof}

\subsection{Explicit bounds and the scalar exponent}

A vector $C$ satisfying $T_mC\leq C$ is a supersolution.
Monotonicity and contraction imply $\beta\leq C$, by iterating $T_m$
from $C$.  We use this observation to bound both the individual
coordinates and their scalar maximum
\begin{equation*}
 S(m)=\max\left\{m+\nu,\max_{0\leq h<\nu}\beta_h(m)\right\}.
\end{equation*}
This maximum combines the diagonal term with the weighted
intersection layers, and hence bounds the exponent of their sum in
the first Cauchy--Schwarz estimate.
For $1\leq j<\nu$, define
\begin{equation}\label{eq:H-interior}
 H_j(m)=\max\left\{
 m+b_j+j,\;\frac{3m+\Lambda_j}{2}+j,\;2m+j-1
 \right\},
\end{equation}
and define the boundary expressions by
\begin{align*}
 H_0(m)&=\max\left\{m+b_0,\frac{3m+\Lambda_0}{2}\right\},\\
 H_\nu(m)&=\max\{m+\nu,2m+\nu-1\}.
\end{align*}
The three terms in \eqref{eq:H-interior} correspond to the diagonal
branch at $j$, the adjacent layer feeding into $j$, and the cap at
$j-1$, respectively.

\begin{proposition}[Evaluation on breakpoint intervals]\label{prop:scalar-evaluation}
For $1\leq j\leq\nu$,
\begin{align}
 S(m)&=H_j(m),
 \qquad b_j\leq m\leq b_{j-1}, \label{eq:scalar-stages}\\
  S(m)&=H_0(m),
 \qquad m\geq b_0. \label{eq:scalar-terminal}
 \end{align}
In particular, on $0\leq m\leq(\nu+1)\delta$,
\begin{equation}\label{eq:F-equals-fixed}
 S(m)=F_{\nu,\delta}(m).
\end{equation}
\end{proposition}

\begin{proof}
\emph{Upper bounds.}
Fix $0\leq j<\nu$, put $B=H_j(m)$, and define
\begin{equation}\label{eq:C-supersolution}
 C_h=
 \begin{cases}
  2m+h,&h<j,\\
  B-j(h-j),&h\geq j.
 \end{cases}
\end{equation}
The definition of $H_j$ gives
\begin{equation}\label{eq:H-constraints}
 B\geq m+b_j+j,
 \qquad 2B\geq3m+\Lambda_j+2j,
\end{equation}
and $B\geq2m+j-1$ if $j\geq1$.  Also
\[
 \overline C_\nu=m+\nu\leq B-j(\nu-j),
\]
because $b_j\geq\delta(\nu-j)\geq(j+1)(\nu-j)$.
Thus the upper boundary is bounded by the same affine expression
as the internal coordinates $k\geq j$.

For $h<j$, the cap gives $(T_mC)_h\leq C_h$.  For $h\geq j$,
\eqref{eq:coefficient-upper} and \eqref{eq:H-constraints} give
\[
 m+\overline C_k+\kappa_{kh}+2h-k
 \leq m+B+b_j+j-2j(h-j)\leq2C_h
 \qquad(k\geq j).
\]
Similarly, \eqref{eq:coefficient-lower} gives
\[
 m+\overline C_k+\kappa_{kh}+2h-k
 \leq3m+\Lambda_j+2j-2j(h-j)\leq2C_h
 \qquad(k<j).
\]
Hence $T_mC\leq C$.  Every coordinate of $C$, as well as $m+\nu$,
is at most $B$, so $S(m)\leq H_j(m)$.  The bound
$S(m)\leq H_\nu(m)$ follows directly from the caps.

\medskip
\noindent\emph{Matching lower bounds.}
Keeping the diagonal branch in the fixed-point equation gives
\begin{equation}\label{eq:self-loop-lower}
 \beta_h\geq\min\{2m+h,m+b_h+h\}.
\end{equation}
Indeed, if $\beta_h<2m+h$, the recursive branch is active and the
$k=h$ term can be rearranged to give $\beta_h\geq m+b_h+h$.
For $b_j\leq m\leq b_{j-1}$ with $1\leq j<\nu$, this implies
\[
 \beta_j\geq m+b_j+j,
 \qquad \beta_{j-1}=2m+j-1.
\]
The adjacent branch $k=j-1$ now gives
\[
 \beta_j\geq\frac{3m+\Lambda_j}{2}+j,
\]
since $\Lambda_j\leq b_j\leq m$ ensures that the cap does not
truncate this value.  These are the three terms of $H_j(m)$.
On $0\leq m\leq b_{\nu-1}=\delta$, the term $m+\nu$ and
\eqref{eq:self-loop-lower} at $h=\nu-1$ give $S(m)=H_\nu(m)$.
For $m\geq b_0$, the diagonal and $k=-1$ branches at $h=0$ give
\[
 \beta_0\geq m+b_0,
 \qquad \beta_0\geq\frac{3m+\Lambda_0}{2}.
\]
Neither value exceeds the cap $2m$, since $\Lambda_0\leq b_0\leq m$.
Thus $S(m)=H_0(m)$.  The ordering in \eqref{eq:threshold-order}
proves \eqref{eq:scalar-stages}--\eqref{eq:scalar-terminal}.

\medskip
\noindent\emph{Identification with the polygonal curve.}
If $w_j=1$, then $\Lambda_j\leq b_j-1$, so the middle term of
$H_j$ is at most the average of the first and third terms.  Those
two lines meet at $m=b_j+1$.  If $w_j=2$, then $\Lambda_j=b_j$;
for $m\geq b_j$, the middle term dominates the first and meets the
third at $m=b_j+2$.  Thus, on $[b_j,b_{j-1}]$, the graph has the
vertices $A_j,B_j,A_{j-1}$ from \eqref{eq:AB-def}.
The same description holds initially because
$b_\nu=0$, $w_\nu=1$, and $b_{\nu-1}=\delta$.
Finally, the two terms of $H_0$ meet at
$m=2b_0-\Lambda_0=b_\ast$ and are the two terminal pieces of $F$.
This proves \eqref{eq:F-equals-fixed}.
\end{proof}

The supersolutions give explicit bounds for each intersection layer,
not just for the scalar maximum.

\begin{corollary}[Explicit bounds for intersecting pairs]
\label{cor:explicit-pair-bounds}
Let $\delta>\nu\geq1$, $M=q^m$, and $0\leq h<\nu$.
If $1\leq j<\nu$ and $b_j\leq m\leq b_{j-1}$, then
\begin{equation}\label{eq:explicit-stage-pairs}
 N_{h,\nu,\delta}(M)
 \ll q^{\min\{2m,\,F_{\nu,\delta}(m)+j^2-(j+1)h\}}.
\end{equation}
If $b_0\leq m\leq(\nu+1)\delta$, then
\begin{equation*}
 N_{h,\nu,\delta}(M)\ll q^{F_{\nu,\delta}(m)-h}.
\end{equation*}
On $0\leq m\leq\delta$, the trivial bound $N_{h,\nu,\delta}(M)\leq M^2$
applies.  Bounds in the original dimensions follow by replacing
$h$ with $k-\rho$, as in \eqref{eq:pair-count-original}.
\end{corollary}

\begin{proof}
On the $j$th interval, \eqref{eq:C-supersolution} gives
$\beta_h\leq F_{\nu,\delta}(m)-j(h-j)$ for $h\geq j$.
Combine this with \cref{thm:pair-count-fixed} and the trivial cap.
For $h<j$, the second exponent in \eqref{eq:explicit-stage-pairs}
is at least $2m$, since $F_{\nu,\delta}(m)\geq2m+j-1$.
On the terminal interval use the constant supersolution $j=0$.
\end{proof}

The coordinatewise fixed-point estimate
\eqref{eq:pair-count-fixed-bound} can be stronger than these explicit
bounds.  Also, the equality \eqref{eq:F-equals-fixed} evaluates the
recursive system; it does not assert that a family of flats attains
every fixed-point coordinate.

\subsection{Application to point--flat incidences}

\begin{proof}[Proof of \cref{thm:main}]
If $\nu=0$, then $d=1$ and $\rho+F_{0,\delta}(m)=n+m$, so the
claim is \cref{thm:vinh}. Suppose that $\nu\geq1$.
For $m>(\nu+1)\delta$, the affine continuation in the statement gives
\[
 \rho+F_{\nu,\delta}(m)-(m+nd)
 =\nu+\tfrac12\bigl(m-(\nu+1)\delta\bigr)>0.
\]
The minimum in \eqref{eq:main-incidence} is therefore $m+nd$, and
\cref{thm:vinh} again proves the claim. We may thus assume
$0\leq m\leq(\nu+1)\delta$. Put
$p=\theta_n/\theta_{n+d}$.  By \cref{prop:first-cs,lem:duality},
\begin{align*}
 \bigl|I(P,\cL)-p|P|M\bigr|^2
 &\ll |P|q^\rho\left(q^{m+\nu}
       +\sum_{h=0}^{\nu-1}q^hN_{h,\nu,\delta}(M)\right)\\
 &\ll |P|q^\rho\left(q^{m+\nu}
       +\sum_{h=0}^{\nu-1}q^{\beta_h(m)}\right)\\
 &\ll |P|q^{\rho+F_{\nu,\delta}(m)},
\end{align*}
where the last two steps use \cref{thm:pair-count-fixed} and
\eqref{eq:F-equals-fixed}.

To replace $p$ by $q^{-d}$, note that
\begin{equation*}
 F_{\nu,\delta}(m)\geq2m+\nu-\delta,
 \qquad 0\leq m\leq(\nu+1)\delta.
\end{equation*}
Before $b_0$ we have the stronger inequality $F(m)\geq2m$;
afterward use
$F(m)\geq(3m+\Lambda_0)/2$ and
$\Lambda_0=\delta(\nu-1)+2\nu$.
Since $n+d=\nu+\delta$, $|P|\ll q^{n+d}$, and
$|p-q^{-d}|\ll q^{-n-d-1}$,
\[
 |p-q^{-d}|\,|P|M
 \ll |P|^{1/2}q^{m-(\nu+\delta)/2-1}
 \ll |P|^{1/2}q^{(\rho+F_{\nu,\delta}(m))/2}.
\]
Taking the minimum with \cref{thm:vinh} proves the theorem.
\end{proof}

\subsection{The two-term corollaries}

\begin{proof}[Proof of \cref{cor:standard,cor:dual-standard}]
In either corollary, the case $i=0$ follows from
\cref{thm:kong-tamo}.  If $i=n$ in \cref{cor:standard}, or
$i=d-1$ in \cref{cor:dual-standard}, the first term is the Vinh
bound.  We may therefore exclude these cases.

For $1\leq i<\nu$ with $i\leq\delta-\nu$, put $j=\nu-i$.
Then $j\geq r$, so $c_j=0$, $w_j=1$, and $b_j=i\delta$.
The middle term of $H_j$ is redundant, and
\cref{prop:scalar-evaluation} gives
\begin{equation}\label{eq:nonexceptional-stage-bound}
 \rho+F_{\nu,\delta}(m)
 \leq\rho+H_j(m)
 =\max\{m+i\delta+n-i,\,2m+n-i-1\}.
\end{equation}
Taking square roots in the incidence bound yields
\[
 \incdev{P}{\cL}
 \ll |P|^{1/2}\left(
 q^{(i\delta+n-i)/2}M^{1/2}
 +q^{(n-i-1)/2}M\right).
\]

For \cref{cor:standard}, the remaining indices have $d>n$,
$\nu=n$, and $\delta=d$, so this is the required estimate.
For \cref{cor:dual-standard} with $d\leq n$, use
$\nu=d-1$ and $\delta=n+1$; then
$i\delta+n-i=(i+1)n$.
If instead $d>n$, the only remaining possibility is $d=n+1$ and
$i=1<\nu=n$.  Again $\delta=n+1$ and
\eqref{eq:nonexceptional-stage-bound} gives the same estimate.
\end{proof}

\section*{Use of artificial intelligence}

The authors used OpenAI's ChatGPT (GPT-5.6 and GPT-6.0)
primarily for assistance with writing and with choosing parameters
in some proofs in the applications section (\cref{sec:applications}) and in
\cref{sec:main-proof}. The basic approach and the remaining proofs were
developed entirely by the authors. The authors reviewed and
verified the AI-assisted material and take full responsibility
for the content of the paper.

\section*{Acknowledgements}
Tao Zhang is partially supported by the National Natural Science Foundation of China (Grant No.~12571357) and the Natural Science Basic Research Program of Shaanxi (Program No.~2025JC-YBMS-048).

\bibliographystyle{alpha}
\bibliography{FFIncidences}

\end{document}